\documentclass{article}

\usepackage{amssymb,amsthm}
\usepackage{amsmath}
\usepackage{latexsym}
\usepackage{amssymb}
\usepackage[dvips]{graphics}
\usepackage[dvips]{graphicx}
\newtheorem{theorem}{Theorem}[section]

\newtheorem{lemma}[theorem]{Lemma}
\newtheorem{corollary}[theorem]{Corollary}
\newtheorem{definition}[theorem]{Definition}
\newtheorem{remark}[theorem]{Remark}
\newtheorem{example}[theorem]{Example}

\newcommand{\R}{{\mathbf R}}
\newcommand{\Z}{{\mathbf Z}}
\newcommand{\Q}{{\mathbf Q}}

\renewcommand{\int}{\rm Int}
\newcommand{\E}{\mathbb E}
\newcommand{\LL}{\mathcal L}

\newcommand{\PP}{\mathbb {P}}
\newcommand{\lk}{{\rm {Lk}}}

\newcommand{\cd}{{\rm cd}}
\newcommand{\gd}{{\rm gdim}}

\graphicspath{%
    {converted_graphics/}
    {/}
}
\begin{document}

\title{Fundamental groups of clique complexes of random graphs}         
\author{Armindo Costa, Michael Farber and Danijela Horak}        
\date{November 15, 2014}          
\maketitle

\abstract{We study fundamental groups of clique complexes associated to random Erd\"os - R\'enyi graphs $\Gamma$. 
We establish thresholds for a number of properties of fundamental groups of these complexes $X_\Gamma$. In particular, if $p=n^\alpha$, we show that 
\begin{equation*}
\begin{array}{lll}
\gd(\pi_1(X_\Gamma))=\cd(\pi_1(X_\Gamma))=1, &\mbox{if}& \alpha <-1/2,\\
\gd(\pi_1(X_\Gamma))=\cd(\pi_1(X_\Gamma)) =2, & \mbox{if} & -1/2 <\alpha <-11/30,\\
\gd(\pi_1(X_\Gamma))=\cd(\pi_1(X_\Gamma)) =\infty,& \mbox{if}&-11/30 <\alpha <-1/3,
\end{array}
\end{equation*}
a.a.s., where $\gd$ and $\cd$ denote the geometric dimension and cohomological dimension correspondingly. 
It is known that the fundamental group $\pi_1(X_\Gamma)$ is trivial for $\alpha >-1/3$. 
We prove that for $-11/30 <\alpha<-1/3$ the fundamental group $\pi_1(X_\Gamma)$ has 2-torsion but has no $m$-torsion for any given prime $m\ge 3$. 
We also prove that aspherical subcomplexes of the random clique complex $X_\Gamma$ satisfy the Whitehead Conjecture, i.e. all their subcomplexes are also aspherical, a.a.s.\footnote{The symbol a.a.s. stands for \lq\lq asymptotically almost surely\rq\rq\,  which means that the 
probability that the corresponding statement holds tends to 1 as $n$ tends to infinity.}
}

\section{Introduction} 

A clique in a graph $\Gamma$ is a set of vertices of $\Gamma$ such that any two of them are connected by an edge. 
The family of cliques of $\Gamma$ forms a simplicial complex $X_\Gamma$ with the vertex set $V(X_\Gamma)$ equal the vertex set $V(\Gamma)$ of $\Gamma$. 
The complex $X_\Gamma$ is called {\it the clique complex} (or {\it the flag complex)} of $\Gamma$. Clearly, the 1-skeleton of $X_\Gamma$ is the graph $\Gamma$ itself. 

In this paper we consider the clique complexes $X_\Gamma$ of random Erd\H{o}s - R\'enyi graphs $\Gamma\in G(n,p)$. 
Recall that $G(n,p)$ is the probability space of all subgraphs $\Gamma$ of the complete graph on $n$ vertices satisfying  $V(\Gamma)=\{1, \dots, n\}$, where 
the probability of a graph $\Gamma$  equals
$${\mathbb P}(\Gamma) \, =\, p^{e(\Gamma)}(1-p)^{{\binom n 2}-e(\Gamma)}.$$ Here $p\in (0,1)$ is a probability parameter, which in general is a function of $n$, 
and $e(\Gamma)$ denotes the number of edges in $\Gamma$. 
The complex $X_\Gamma$, where $\Gamma\in G(n,p)$, is a random simplicial complex. One is interested in topological properties of $X_\Gamma$ which are satisfied with high probability when the number of vertices $n$ tends to infinity.

Topology of clique complexes of random graphs were studied by M. Kahle et al. in a series of papers \cite{Kahle1}, \cite{Kahle3}, \cite{KM}, \cite{KP}. A recent survey is given in \cite{Ksurvey}. The following result is stated in a simplified form. 

\vskip 0.3cm
\noindent
{\bf Theorem:} {\rm [See M. Kahle \cite{Kahle1}, Theorems 3.5 and 3.6] }
{\it Consider the clique complex $X_\Gamma$ of a random graph $\Gamma\in G(n,p)$ where $p=n^\alpha$. Let $k>0$ be a fixed integer. Then 
\begin{enumerate}
\item[{\rm (a)}] If $ \alpha < -1/k$ then $H_k(X_\Gamma;\Z)=0$ a.a.s.
\item [{\rm (b)}] If $-1/k < \alpha < -1/(k+1)$ then $H_k(X_\Gamma;\Q)\not=0$, a.a.s.
\end{enumerate}
}

One knows (see Theorem 3.4 from \cite{Kahle1})
that {\it for $p=n^\alpha$ the random clique complex $X_\Gamma$ is $k$-connected a.a.s. if} $$\alpha > -(2k+1)^{-1}.$$  In particular, 
the clique complex $X_\Gamma$ is connected for $\alpha >-1$ and it is simply connected for $\alpha>-1/3$, a.a.s.

In this paper we are interested in the properties of the fundamental group of a random clique complex and therefore (see above) 
we shall restrict our attention to the regime 
$\alpha <-1/3$ where $p=n^\alpha$. In a recent preprint \cite{Babson} E. Babson proved that for $\epsilon >0$ and $n^{\epsilon -1/2}<p<n^{n-\epsilon -1/3}$ the fundamental group $\pi_1(X_\Gamma)$ is nontrivial and is hyperbolic in the sense of Gromov \cite{Gromov87}.

In this paper we use the notation $f\ll g$ to indicate that $f/g\to 0$ as $n\to \infty$. 

The main results of this paper are as follows:
\vskip 0.3cm
\noindent
{\bf Theorem A:} {\rm [See Theorem \ref{thm1}] {\it 
 If 
\begin{eqnarray}\label{assumption1}p\ll n^{-1/2}\end{eqnarray} then, 
with probability tending to 1 as $n\to \infty$, the clique complex $X_\Gamma$ is simplicially collapsible to a graph. In particular 
 under the above assumption the 
fundamental group $\pi_1(X_\Gamma, x_0)$ of a random 
clique complex $X_\Gamma$, where $\Gamma\in G(n,p)$, 
is free, for any choice of the base point $x_0\in X_\Gamma$, a.a.s.
Moreover,each connected component of the 2-skeleton $X_{\Gamma}^{(2)}$ is homotopy equivalent to a wedge of circles and 2-spheres, a.a.s.
}
\vskip 0.3 cm
Note that in the range (\ref{assumption1}) the dimension of $X_\Gamma$ is $\le 3$ and the 2-skeleton $X_\Gamma^{(2)}$ contains the tetrahedron and its subdivision having 5 vertices, a.a.s. Hence, the 2-skeleton $X_\Gamma^{(2)}$ is not collapsible to a graph. 

Note also that for $p\gg n^{-1/2}$ the fundamental group $\pi_1(X_\Gamma)$ ceases to be free. This follows from a theorem of M. Kahle \cite{Kahle3} which states that 
for $$p^2 \ge (3/2 +\epsilon)\cdot n^{-1}\cdot \log n$$
the fundamental group $\pi_1(X_\Gamma)$ has property (T) and thus its cohomological dimension is $\ge 2$, a.a.s. In the following theorem we describe the range in which the cohomological dimension of 
$\pi_1(X_\Gamma)$ equals 2. 

We wish to mention a recent preprint \cite{ALS} where random triangular groups are studied; this class of random groups is different from the class of fundamental groups of random clique complexes although these two classes of random groups share several common features. The main result of \cite{ALS} states that there exists an interval in which the random triangular group is neither free nor possesses the property T. We expect that such intermediate regime exists in the model which we study in this paper. We shall address this issue elsewhere. 


\vskip 0.3cm
\noindent {\bf Theorem B:} {\rm [See Theorem \ref{cd2}]} { \it
Assume that 
\begin{eqnarray}\label{11/30}
p \ll n^{-11/30}. 
\end{eqnarray}
Then the fundamental group $\pi_1(X_\Gamma)$ of the clique complex of a random graph $\Gamma\in G(n,p)$ satisfies
\begin{eqnarray}
\gd(\pi_1(X_\Gamma)) = \cd(\pi_1(X_\Gamma))\le 2,
\end{eqnarray}
and in particular
$\pi_1(X_\Gamma)$ is torsion free, a.a.s.
Moreover, if for some $\epsilon >0$ one has
$$\left( (3/2+\epsilon)\cdot n^{-1} \cdot \log n\right)^{1/2} \, \le \, p\, \ll \, n^{-11/30}$$
then 
\begin{eqnarray}
\gd(\pi_1(X_\Gamma)) = \cd(\pi_1(X_\Gamma))= 2,
\end{eqnarray}
a.a.s.}

Recall that {\it geometric dimension} $\gd(G)$ of a discrete group $G$ is defined as the minimal dimension of an aspherical CW-complex having $G$ as its fundamental group. 
The {\it cohomological dimension} $\cd(G)$ is the shortest length of a free resolution of $\Z$ viewed as a $\Z[G]$-module. In general $\cd(G)\le \gd(G)$ and a classical theorem of Eilenberg and Ganea \cite{EG} states that $\cd(G)=\gd(G)$, except of three low-dimensional cases. At present it is known that the equality 
$\cd(G)=\gd(G)$ holds, except possibly for the case when $\cd(G)=2$ and $\gd(G)=3$. 
The Eilenberg--Ganea Conjecture states that $\cd(G)=2$ implies $\gd(G)=2$.

The following theorem states that 2-torsion appears in the fundamental group of a random clique complex when we cross the threshold 11/30:

\vskip 0.3 cm\noindent
{\bf Theorem C:} 
{\rm [See Theorem \ref{2torsion}]} {\it 
Assume that 
\begin{eqnarray}
n^{-11/30}\ll p \ll n^{-1/3-\epsilon}
\end{eqnarray}
where $0<\epsilon < 1/30$ is fixed. Then the fundamental group $\pi_1(X_\Gamma)$ has 2-torsion and thus its cohomological dimension and geometric dimension are infinite, a.a.s.
}
\vskip 0.3cm

Surprisingly, odd torsion does not appear in fundamental groups of random clique complexes until the triviality threshold $p=n^{-1/3}$:

\vskip 0.3 cm
\noindent
{\bf Theorem D:} {\rm [See Theorem \ref{oddtorsion}]} {\it 
Let $m\ge 3$ be a fixed prime. Assume that 
\begin{eqnarray}
 p \ll n^{-1/3-\epsilon}
\end{eqnarray}
where $\epsilon >0$ is fixed. Then a random graph $\Gamma\in G(n,p)$ with probability tending to 1 has the following property: the fundamental group of any subcomplex $Y\subset X_\Gamma$ has no $m$-torsion. 
}\vskip 0.3cm

Surprisingly, we see that for all the assumptions on the probability parameter $p$ considered in Theorems A, B, C, the fundamental groups of random clique complexes have cohomological dimension $1, 2$ or $\infty$, which implies that {\it probabilistically} the Eilenberg--Ganea conjecture is satisfied. 
Note also that in the complementary range, when $p=n^\alpha$ with $\alpha>-1/3$, the clique complex $X_\Gamma$ of a random graph $\Gamma$ is simply connected a.a.s. (by Theorem 3.4 from \cite{Kahle1}) and hence the Eilenberg--Ganea conjecture is also {\it probabilistically} satisfied. 
We may also mention here that any finitely presented group appears as the fundamental group of a clique complex $X_\Gamma$ for a graph 
$\Gamma\in G(n,p)$ with any $n$ large enough. 

Next we state a result in the direction of the Whitehead conjecture. 

Recall that a connected simplicial complex $Y$ is said to be aspherical if $\pi_i(Y)=0$ for all $i\ge 2$; this is equivalent to the requirement that the universal cover of $Y$ is contractible. 
For 2-dimensional complexes $Y$ the asphericity is equivalent to the vanishing of the second homotopy group $\pi_2(Y)=0$, or equivalently, that any continuous map 
$S^2\to Y$ is homotopic to a constant map. 
Random aspherical 2-complexes could be helpful for testing probabilistically the open problems of two-dimensional topology, such as the Whitehead conjecture. 
This conjecture stated by J.H.C. Whitehead in 1941 claims that a subcomplex of an aspherical 2-complex is also aspherical. 
Surveys of results related to the Whitehead conjecture can be found in \cite{B}, \cite{R}. 

\vskip 0.3 cm
\noindent
{\bf Theorem E:} {\rm [See Corollary \ref{wc}]} {\it 
Assume that $p\ll n^{-1/3-\epsilon}$, where $\epsilon >0$  is fixed. Then, for a random graph $\Gamma\in G(n,p)$, the clique complex $X_\Gamma$ has the following property with probability tending to 1 as $n\to \infty$: any aspherical subcomplex $Y\subset X_\Gamma^{(2)}$ satisfies the Whitehead Conjecture, i.e. any subcomplex $Y'\subset Y$ is also aspherical. 
}
\vskip 0.3cm

Thus we see that {\it probabilistically}, for large {\it finite} simplicial complexes the Whitehead conjecture holds for all their aspherical subcomplexes. 

We also remark that, as is well known, any finite simplicial complex is homeomorphic to a clique complex $X_\Gamma$ with $\Gamma\in G(n,p)$ for any $n$ large enough; thus every 2-dimensional finite simplicial complex appears up to homeomorphism with positive probability for large $n$.

Recall that a well known result of Bestvina and Brady \cite{BB} states that either the Whitehead conjecture or the Eilenberg--Ganea conjecture must be false. However, Bestvina and Brady consider these two conjectures in a larger class of {\it infinite} simplicial complexes and not necessarily finitely presented groups.

We should also point out the limitations of our approach. We have no results in the direction of the Eilenberg--Ganea conjecture near the two critical
values of the probability parameter $p=n^{-11/30}$ and $p=n^{-1/3}$. Besides, we do not know the validity of the probabilistic version of the Whitehead conjecture (the analogue of Theorem E) for $p\gg n^{-1/3}$.


A few words about terminology we use in this paper.  
By a {\it 2-complex} we understand a finite simplicial complex of dimension $\le 2$. 
The $i$-dimensional simplexes of a 2-complex are called {\it vertices} (for $i=0$), {\it edges} (for $i=1$) and {\it faces} (for $i=2$). 

A 2-complex is said to be {\it pure} if every vertex and every edge are incident to a face. 

The {\it pure part} of a 2-complex is the closure of the union of all faces. 

The {\it  degree} of an edge $e$ of $X$ is the number of faces containing $e$. 

The {\it boundary} $\partial X$ of a 2-complex $X$ is the union of all edges of degree one. 
We say that a 2-complex $X$ is {\it closed}
if $\partial X=\emptyset$.

We denote by $V(X), E(X), F(X)$ the sets of vertices, edges and faces of $X$, correspondingly. 
We also use the notations $v(X)=|V(X)|,$ $e(X)=|E(X)|,$ $f(X)=|F(X)|$. 

A connected 2-complex $X$ is {\it strongly connected} if $X-V(X)$ is connected. 

We use the notations $P^2$ for the real projective plane. 

\vskip 0.5cm

The authors thank the referee for making useful critical remarks.

\section{The containment problem}

In this section we collect some known results which we shall use in this paper. The only new result here is Theorem \ref{thmbalanced} 
which describes properties of clean triangulations of surfaces. 

Let $S$ be a 2-complex. 
We have already introduced the notations 
\begin{eqnarray}\label{nudef}\nu(S)=\frac{v(S)}{e(S)}, \quad \quad \tilde \nu(S) = \min_{S'\subset S} \nu(S')\end{eqnarray} 
where $v(S)$ and $e(S)$ denote the number of vertices and edges in $S$. 
Although the numbers $\nu(S)$ and $\tilde \nu(S)$ depend only on the 1-skeleton of $S$, it is convenient to think about $\nu(S)$ and $\tilde \nu(S)$ as being associated to the whole 2-complex $S$ due to the following formula
\begin{eqnarray}\label{useful}
\nu(S) = \frac{1}{3} + \frac{3\chi(S) +L(S)}{3e(S)}.
\end{eqnarray}
Here 
\begin{eqnarray*}L(S)&=&\sum_{e}\left[2-\deg(e)\right]\\
&=&2e(S) -3f(S).\end{eqnarray*}
In the definition of $L(S)$ the sum is over all the edges $e$ of $S$ and $\deg(e)$ is the number of faces containing $e$. 
Note that $L(S) \le 0$ assuming that $S$ is closed, i.e. if $\deg(e)\ge 2$ for every edge $e$ of $S$.

The embeddability of $S$ into $X_\Gamma$ is equivalent to the embeddability of the 1-skeleton of $S$ into $\Gamma$. 
The following result follows from the well known subgraph containment problem in random graph theory, see \cite{JLR}, Theorem 3.4 on page 56. 
\begin{theorem}\label{embed}
Let $S$ be a fixed finite simplicial complex. Consider the clique complex $X_\Gamma$ associated to a random Erd\H{o}s - R\'enyi graph $\Gamma\in G(n,p)$. Then:
\begin{enumerate}
\item[(A)] If $p\ll n^{-\tilde \nu(S)}$ then the probability that $S$ admits a simplicial embedding into $X_\Gamma$ tends to $0$ as $n\to \infty$;
\item [(B)] If $p\gg n^{-\tilde \nu(S)}$ the the probability that $S$ admits a simplicial embedding into $X_\Gamma$ tends to $1$ as $n\to \infty$;
\end{enumerate}
\end{theorem}

\begin{definition} A graph $\Gamma$ is said to be balanced if for any proper subgraph $\Gamma'\subset \Gamma$ one has $\nu(\Gamma)\le \nu(\Gamma')$. 
A graph $\Gamma$ is said to be strictly balanced if for any proper subgraph $\Gamma'\subset \Gamma$ one has $\nu(\Gamma) < \nu(\Gamma')$. 
\end{definition}
\begin{definition}
A simplicial 2-complex $S$ is said to be $\nu$-balanced (or strictly $\nu$- balanced) if its 1-skeleton is balanced (or strictly balanced, correspondingly).
\end{definition}

\begin{definition}
A simplicial 2-complex $S$ is called clean if it coincides with the 2-skeleton of the clique complex of its 1-skeleton.
\end{definition}

In other words, a triangulation is clean if any clique consisting of three vertices spans a 2-simplex. 

\begin{example} {\rm 
Let $K_{r+1}$ be the complete graph on $r+1$ vertices. It is easy to see that it is strictly $\nu$-balanced and 
$$\tilde \nu(K_{r+1}) = \nu(K_{r+1}) = \frac{2}{r}.$$}
\end{example}

As a corollary of Theorem \ref{embed} we obtain:
\begin{corollary} \label{dimension}
If the probability parameter $p$ satisfies 
$$n^{-2/r} \ll p \ll n^{-2/(r+1)}$$
where  $r\ge 2$ is an integer, then the dimension 
$\dim X_\Gamma $ of a random clique complex $X_\Gamma$ equals $r$, a.a.s.
\end{corollary}

\begin{example}\label{ex1}{\rm 
Consider a triangulated surface $S$ having a vertex $x$ of degree 3. 
Clearly such triangulation is not clean. Assume that either $S$ is orientable and has genus $>1$ or it is non-orientable and has genus $>2$. 
If $\Gamma$ denotes the 1-skeleton of $S$ then $\nu(S)=\nu(\Gamma)<1/3$. 
Removing the vertex $x$ and the three incident to it edges we obtain a graph $\Gamma'\subset \Gamma$ with $v(\Gamma')=v(\Gamma)-1$ and $e(\Gamma') =e(\Gamma)-3$. 
Since $\nu(\Gamma) <1/3$ we see that 
$$\nu(\Gamma') = \frac{v(\Gamma)-1}{e(\Gamma)-3}<\nu(\Gamma),$$
i.e. $\Gamma$ is not $\nu$-balanced. 
}
\end{example}

The following Theorem is analogous to Theorem 27 from \cite{CCFK}. 

\begin{theorem}\label{thmbalanced} Any clean triangulation of a closed connected surface $S$ with $\chi(S)\ge 0$ is $\nu$-balanced. Moreover, 
if $\chi(S)>0$ then any clean triangulation of $S$ is strictly $\nu$-balanced. 
\end{theorem}
\begin{proof} Let $\Gamma$ be a graph such that $S$ is the clique complex $S=X_\Gamma$. Let $\Gamma'\subset \Gamma$ be a proper subgraph and let $S'=X_{\Gamma'}$ denotes the clique complex of $\Gamma'$. 
Without loss of generality we may assume that $\Gamma'$ is connected. 
Due to formula (\ref{useful}), the inequality $\nu(S)< \nu(S')$ would follow from 
\begin{eqnarray}\label{interm}
3\chi(S') +L(S') \ge 3 \chi(S),
\end{eqnarray}
since $L(S)=0$, $e(S)>e(S')$ and $\chi(S)\ge 0$.
From now on all homology and cohomology group will have coefficient group $\Z_2$ which will be omitted from the notation. Besides, we will use the symbol 
$b_i'(X)$ to denote $\dim_{\Z_2}H_i(X)$, the $i$-th Betti number with $\Z_2$ coefficients. 
Consider the exact sequence 
\begin{eqnarray}\label{exsec}0\to H_2(S) \to H_2(S,S') \stackrel{j_\ast}\to H_1(S') \to H_1(S) \to H_1(S, S') \to 0.\end{eqnarray}
Here we used that $H_2(S')=0$ (since $S'$ is a proper subcomplex of $S$) and $H_2(S)=\Z_2$. By Poincar\'e duality, the dimension of $H_2(S, S')$ equals $\dim H^0(S-S')=k$, the number of path-connected components of the complement
$S-S'$, see Proposition 3.46 from \cite{Hatcher}. Thus, (\ref{exsec}) implies the inequality
\begin{eqnarray}
\label{ineq1}
b_1'(S)\ge b_1'(S') -k+1.
\end{eqnarray}
Substituting $\chi(S)=2-b_1'(S)$, $\chi(S')=1-b_1'(S')$ into (\ref{interm}) we see that (\ref{interm}) would follows from (\ref{ineq1}) once we show that
$L(S') \ge 3k$. Note that $L(S')=e_1(S')+2e_2(S')$ where $e_i(S')$ denotes the number of edges of $S'$ which have degree $i$, where $i=0,1$. 
If $C_1, \dots, C_k$ denote the boundary circles of the connected components of $S-S'$ then one has 
$$\sum_{j=1}^k |C_j| = e_1(S') +2e_0(S')=L(S'),$$ since each edge of $S'$ having degree one belong to exactly one of the circles $C_j$ and each edge of degree zero belongs to two circles $C_j$. Clearly, $|C_j|\ge 3$ for each $C_j$ and 
the inequality $L(S') \ge 3k$ follows. 
\end{proof}

For a triangulation $S$ of a compact orientable surface $\Sigma_g$ of genus $g$ one has using the formula (\ref{useful}),
\begin{eqnarray}\label{orientable} 
\nu(S)= \frac{1}{3}+\frac{2-2g}{e(S)}.
\end{eqnarray}
Similarly, for a triangulation $S$ of a compact non-orientable surface $N_g$ of genus $g$ one has
\begin{eqnarray}\label{nonorientable}
\nu(S)= \frac{1}{3}+\frac{2-g}{e(S)}.
\end{eqnarray}
Thus we see that $\nu(S) <1/3$ if $S$ is orientable and $g>1$ or if $S$ is non-orientable and $g>2$.

\begin{remark} {\rm It is easy to show that the assumption $\chi(S) \ge 0$ of Theorem \ref{thmbalanced} is necessary. 
More specifically, {\it any closed surface with $\chi(S)<0$ admits a non-$\nu$-balanced clean triangulation.}

Indeed, let $S$ be a clean triangulation of a surface with $\chi(S) <0$; then 
$\nu(S) <1/3$ 
(by (\ref{orientable}) and (\ref{nonorientable})). Let $X\subset S$ be the subcomplex obtained from $S$ by removing an edge $e\subset S$ and the interiors of two adjacent to $e$ 2-simplexes. 
Then 
$$\nu(X)=\frac{v(S)}{e(S)-1}=\frac{e(S)/3+\chi(S)}{e(S)-1}\leq \frac{e(S)/3-1}{e(S)-1}<\frac{1}{3}$$
Let $D$ be a clean triangulated disc with $r$ interior vertices and whose boundary is a simplicial circle with 4 vertices and 4 edges. For any triangulated disc we have (using the Euler - Poincare formula), 
$$e(D)=2v(D)+r-3$$
and since $v(D) = r+4$ we obtain 
$$e(D) = 3r+5.$$
Let $S'$ be the result of gluing $D$ to $X$ with the identification $\partial D=\partial X$. Obviously $S'$ is homeomorphic to $S$. One has 
$$v(S')=v(S)+r \quad \mbox{and}\quad
e(S')=e(X)+e(D)-e(\partial D)=e(X)+3r+1=e(S)+3r.$$
Hence
$$\nu(S')=\frac{v(S)+r}{e(S)+3r}\to\frac{1}{3}$$
tends to $1/3$ as $r\to \infty$. 
Thus, by taking $r$ large enough we shall have $\nu(S')>\nu(X)$. The obtained triangulation $S'$ is clean and unbalanced since $X$ is a subcomplex of $S'$.
}\end{remark}

\begin{remark}{\rm
Theorem 27 from \cite{CCFK} (which is similar to Theorem \ref{thmbalanced})
is valid under an additional assumption $\chi(S) \ge 0$ which is missing in its statement. 
The assumption $\chi(S)\ge 0$ is essential since {\it any closed surface with negative Euler characteristic $\chi(S)<0$ admits a not $\mu$-balanced triangulation. }
} 
\end{remark}
%

\section{Threshold for collapsibility to a graph}

In this section we prove Theorem A which we restate below:

\begin{theorem}\label{thm1} If 
\begin{eqnarray}\label{assumption}p\ll n^{-1/2}\end{eqnarray} then, 
with probability tending to 1 as $n\to \infty$, the clique complex $X_\Gamma$ is simplicially collapsible to a graph, a.a.s. In particular the 
fundamental group $\pi_1(X_\Gamma, x_0)$ of a random 
clique complex $X_\Gamma$, where $\Gamma\in G(n,p)$, 
is free, for any choice of the base point $x_0\in X_\Gamma$. 
Moreover, under the above assumptions each connected component of the 2-skeleton $X_{\Gamma}^{(2)}$ is homotopy equivalent to a wedge of circles and 2-spheres, a.a.s.
\end{theorem}

The proof of Theorem \ref{thm1} uses a deterministic combinatorial assertion described below as Theorem \ref{thm2}. In its statement we use the notation 
\begin{eqnarray}\label{rho}\nu(S)=\frac{v(S)}{e(S)}\end{eqnarray}
where $S$ is a simplicial 2-complex and $v(S)$ and $e(S)$ denote the number of its vertices and edges. 
We will also use the invariant
\begin{eqnarray}
\tilde \nu(X) = \min_{S\subset X} \nu(S),
\end{eqnarray}
where $S$ runs over all subcomplexes of $X$. 

We shall denote by ${\cal S}_1$ the tetrahedron (the 2-complex 
homeomorphic to the sphere $S^2$ and having 
4 vertices, 6 edges and 4 faces) and by ${\cal S}_2$ the triangulation of $S^2$ having 5 vertices, 9 edges and 6 faces. 
Clearly, $\nu({\cal S}_1)=2/3>1/2$ and $\nu({\cal S}_2)=5/9>1/2$. 
The complexes ${\cal S}_1$ and ${\cal S}_2$ play a special role in our study: Theorem \ref{thm2} below implies that {\it any closed 2-complex $X$ satisfying 
$\tilde \nu(X)> 1/2$ contains either ${\cal S}_1$ or ${\cal S}_2$ as a simplicial subcomplex.} 

\begin{theorem}\label{thm2}
There exists an infinite set $\cal L$ of isomorphism types of finite simplicial 2-complexes satisfying the following properties:
\begin{enumerate}
  \item[{(1)}] for any $S\in \cal L$ one has $\nu(S)\le 1/2$;
  \item[{(2)}] the set $\cal L$ has at most exponential size in the following sense: for an integer $E$ let ${\cal L}_E$ denote the set $\{S\in {\cal L}; e(S)\le E\}$. 
  Then for some positive constants $A$ and $B$ one has 
  $$|{\cal L}_E|\le A\cdot B^E,$$
where $A$ and $B$ are independent of $E$;
  \item[{(3)}] any closed pure 2-complex $X$ contains a simplicial subcomplex isomorphic to some $S\in {\cal L}\cup \{{{\cal S}}_1, {{\cal S}}_2\}$.  
  \end{enumerate}
\end{theorem}


Property (3) is the main universal feature of the set $\cal L$. 

\begin{proof}[Proof of Theorem \ref{thm2}]
We start with a few remarks: 

For a triangulated 2-disc $X$ having $v$ vertices such that among them there are $v_i$ internal vertices, one has
\begin{eqnarray}\label{disc}
\nu(X) = \frac{v}{2v+v_i -3}.
\end{eqnarray}
Thus one has $\nu(X)=1/2$ for $v_i=3$ and $\nu(X) >1/2$ only for $v_i=0,1,2$. 
Formula (\ref{disc}) follows from the relations $3f=2e-v_\partial$ and $v-e+f=1$ where $v_\partial = v-v_i$ is the number of vertices on the boundary. 

\begin{figure}[h]
\centering
\includegraphics[width=0.25\textwidth]{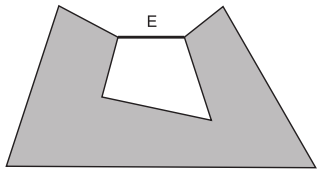}
\caption{External edge $E$.}\label{externaledge}
\end{figure}
The operation of {\it adding an external edge} to a simplicial complex $X$ gives a simplicial complex $X'=X\cup E$ where $E$ is a arc (i.e. a space homeomorphic to $[0,1]$) and $X\cap E= \partial E$, see Figure \ref{externaledge}. 
Clearly $\nu(X') <\nu(X)$.

If $X$ is obtained from a triangulated disc with $v_i$ internal vertices by adding $c$ external edges, then 
\begin{eqnarray}\nu(X)= \frac{v}{2v+v_i+c-3}\end{eqnarray}
and therefore we see that $\nu(X) \le 1/2$ if and only if $v_i+c\ge 3$. 


We denote by ${\cal L}$ the set of isomorphism types of finite simplicial 2-complexes $S$ 
having the following properties:
the pure part $S_0$ of $S$ admits a surjective simplicial map

$f: S'\to S_0$ where: 

(a) $S'$ is a triangulated disc with one internal vertex; 

(b) $f$ is bijective on the set of faces; 

(c) the image of any edge of $S'$ is an edge of $S_0$; 

(d) $v(S')-v(S_0) \le 2;$

(e) the complex $S$ is obtained from its pure part $S_0$ by adding at most $2$ external edges;

(f) and finally we require that 
\begin{eqnarray}\label{half}\nu(S)\le 1/2.\end{eqnarray}

Typical examples of complexes from $\LL$ are given below. 

If $v(S')-v(S_0)=a$ (\lq\lq the vertex defect\rq\rq) and $e(S')-e(S_0) =b$ (\lq\lq the edge defect\rq\rq) then using 
the inequality $e(S')=2v(S')-2$  (which follows from (\ref{disc})) 
we obtain 
that  (\ref{half}) is equivalent to 
\begin{eqnarray}\label{half1}
2a+c\ge b+2,\end{eqnarray}
where $c =0,1,2$ denotes the number of external edges in $S$. 
Since $a\le 2$ and $c\le 2$, the total number of solutions $(a,b,c)$ to (\ref{half1}) is 19.

Next we show that the set $\cal L$ satisfies property (2) of Theorem \ref{thm2}. 
According to W. Brown \cite{Brown}, the number of isomorphism types of triangulations of the disc $S'$ with $v$ vertices having one internal vertex is less than or equal to
$$ \frac{2v-5}{v-1}\cdot \binom {2v-6}{v-2} \le 2\cdot  2^{2v-6}< 4^v;$$
here we use formula (4.7) from \cite{Brown} with $v=m+4$ and $n=1$. This implies that the number of isomorphism types of triangulations of the disc $S'$ with at most 
$v$ vertices and one internal vertex is less than or equal to $$1+4+\dots +4^v < 4/3\cdot 4^v.$$

We want to estimate above the number of elements $S\in \cal L$ satisfying $e(S)\le E$. 
For $S\in \cal L$ with 
$e(S)\le E$, let $f: S'\to S_0$ be a surjective simplicial map as in the definition of $\LL$. Here $S_0$ is the pure part of $S$ and $S$ is obtained from $S_0$ 
by adding $c = \, 0,\, 1,\,  2$ edges. Then using (\ref{half}), we find $v(S_0)\le  e(S_0)/2+1\le E/2$ and 
$$v(S') \le v(S_0)+2 = v(S)+2 \le E/2+2.$$ The complex $S_0$ is obtained from $S'$ by identifying at most 2 pairs of vertices or by identifying a triple of vertices; the identification of vertices determines the identification of edges. 
As we noted above, there are 19 types of quotients.   
Hence we obtain (assuming that $E\ge 6$)
\begin{eqnarray*}|{\cal L}_E| \le 4/3 \cdot 4^{E/2+2}\cdot 19\cdot  (E/2+2)^4 \cdot  (E/2+2)^4.
\end{eqnarray*}
In the above inequality the first factor $(E/2+2)^4$ accounts for the ways of doing identifications of vertices and the second factor $(E/2+2)^4$ accounts for the ways to add 2 additional edges. Since $(E/2+2)^4 \le 4^{E/2+2}$ we see that 
$$
|{\cal L}_E| \le \frac{4^7\cdot 19}{3}\cdot 8^E.$$
This proves that the set $\cal L$ has property (2) of Theorem \ref{thm2}. 

Below we show that the set $\cal L$ has property (3) of Theorem \ref{thm2}.
We start by describing examples of complexes from $\LL$. 

{\bf Example 1}: Triangulated disc with one internal point and two added external edges (i.e. $v_i=1$ and $c=2$).

{\bf Example 2}: Triangulated disc with two internal points and one added external edge (i.e. $v_i=2$ and $c=1$). 

Note that a triangulated disc with $k$ internal points may be obtained as a quotient of a triangulated disc with $k-1$ internal points by identifying two vertices and two adjacent edges on the boundary. This fact is illustrated by Figure \ref{cutting}. 
\begin{figure}[h]
\centering
\includegraphics[width=0.6\textwidth]{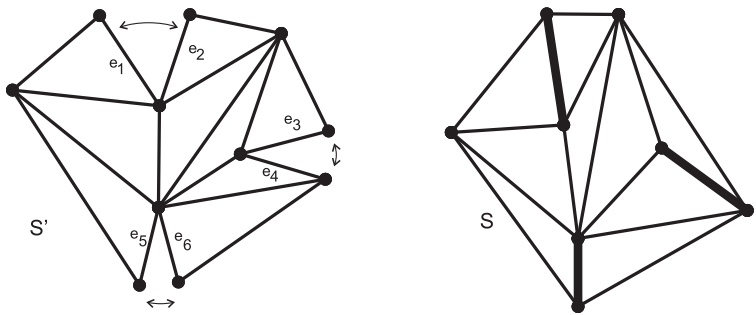}
\caption{Disc with 3 internal points as a quotient of a disc with no internal points; \newline  3 pairs of adjacent edges are identified.}\label{cutting}
\end{figure}


{\bf Example 3}: Consider a simplicial surjective map $f: X'\to X$ where $X'$ is a triangulated disc and $f$ is bijective on faces and every edge of $X'$ is mapped to an edge 
of $X$ and such that $v(X') -v(X)=1$ and $e(X')-e(X)=1$, i.e. exactly two vertices and two (adjacent) edges are identified. 
If $X'$ has $i$ internal vertices then we 
 call such an $X$ {\it a scroll with $i$ internal points}. 
\begin{figure}[h]
\centering
\includegraphics[width=0.25\textwidth]{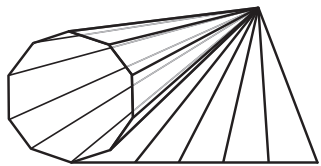}
\caption{Example of a scroll without internal points.}\label{scroll}
\end{figure}

A scroll with two internal points is an element of $\LL$. In particular, every triangulated disc with 3 internal points belongs to $\LL$. 

{\bf Example 4}: A scroll with one internal point and with one external edge added is an element of $\LL$. 

{\bf Example 5}: As above, consider a simplicial surjective map $f: X'\to X$ where $X'$ is a triangulated disc and $f$ is bijective on faces and every edge of $X'$ is mapped to an edge 
of $X$. Assume that exactly two pairs of vertices and two pairs of adjacent edges are identified, i.e. $v(X') -v(X)=2$ and $e(X')-e(X)=2$. 
If $X'$ has one internal vertex then 
$X\in \LL$. We shall call such an $X$ {\it disc with one internal point and with two scrolls}.

%
%
%

Now we show that the set $\cal L$ has property (3) of Theorem \ref{thm2}.
We shall assume the negation of property (3) and arrive to a contradiction. 
Hence, below we assume that there exists a  closed pure 2-complex $X$ which contains no subcomplexes isomorphic to any
$S\in {\cal L}\cup \{{\cal S}_1, {\cal S}_2\}$. 

Consider a vertex $v\in X$ and let $\lk_X(v)$ be the link of $v$ in $X$; it is a graph having no univalent vertices (since $X$ is closed) 
and hence each connected 
component of $\lk_X(v)$ contains a simple cycle $C\subset \lk_X(v)$. 
The cone $D=vC\subset X$ with base $C$ and apex $v$ is a disc with one internal point. 

%
%
%

There may exist at most one {\it external edge}, i.e. an edge $e\subset X$ such that $e\not\subset D$ and $\partial e\subset D$ (since otherwise the union of $D$ and of two such edges would be isomorphic to an element of $\LL$, see Example 1). 
Consider a vertex $w\in C=\partial D$ which is not incident to an external edge (such point exists since $C$ has at least 3 vertices). Let $w', w''\in C$ be the two neighbours of $w$ along $C$. 
The link $\lk_X(w)$ of the vertex $w$ in $X$ is a graph without univalent vertices and the link $\alpha=\lk_D(w)$ is an arc connecting the points $w'$ and $w''$. 
It is obvious that the arc $\alpha$ is contained in a subgraph $\Gamma\subset \lk_X(w)$ which is homeomorphic either to the circle or to one of the two graphs shown in Figure
\ref{specs} (the graph $\Gamma$ can be obtained by extending $\alpha$ in $\lk_X(w)$ until the extension \lq\lq hits itself\rq\rq). 
\begin{figure}[h]
\centering
\includegraphics[width=0.6\textwidth]{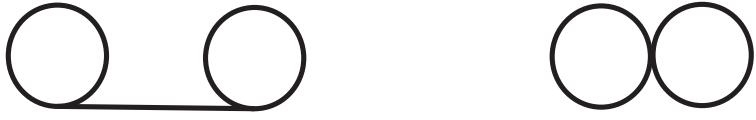}
\caption{Graphs containing the arc $\alpha$.}\label{specs}
\end{figure}
%
Thus the complex $X$ contains the cone $w\Gamma$ over $\Gamma$ with apex $w$. 

The intersection  
$w\Gamma\cap D$ clearly contains $w\alpha$, the cone over the arc $\alpha$. Any vertex $u$ of $(w\Gamma\cap D) - w\alpha$ corresponds to an edge $e\subset X$ such that $\partial e=\{w, u\}$ and $e\not\subset D$. By construction, we know that there are no such external edges. Thus, we see that the set of vertices of $w\Gamma\cap D$ coincides with the set of vertices of $w\alpha$. 

In the case when $\Gamma$ is homeomorphic to one of the graphs shown in Figure \ref{specs} the union $w\Gamma\cup D\subset X$ is a disc with one internal point and with two scrolls which is impossible due to Example 5. Thus the only remaining possibility is that $\Gamma$ is a simple circle. 
The union $w\Gamma\cup D$ can be the tetrahedron ${\cal S}_1$ (iff $\Gamma-\int (\alpha)$ is a single edge contained in $\partial D$); otherwise the union 
$w\Gamma\cup D$ is
a disc. 
The first possibility contradicts our assumptions (we know that $X$ does not contain ${\cal S}_1$ as a subcomplex), 
therefore the union $D_1=w\Gamma\cup D\subset X$ is a disc with two internal points $v, w$. 

Next we repeat the above arguments applied to $D_1\subset X$ instead of $D\subset X$. Consider a point $w_1\in \partial D_1$ and its two neighbours $w'_1, w''_1\in C_1=\partial D_1$. The link $\lk_X(w_1)$ is a graph without univalent vertices and $\alpha_1=\lk_D(w_1)$ is an arc connecting the points $w'_1, w''_1$. 
The arc $\alpha_1$ is contained in a subgraph $\Gamma\subset \lk_X(w_1)$ which is homeomorphic either to the circle or to one of the graphs shown in 
Figure \ref{specs}. The set of vertices of $\Gamma$ contained in $D_1$ coincides with the set of vertices of $\alpha_1$ (since otherwise $X$ would contain a disc with two internal points and with one external edge which contradicts  Example 2). 
In the case when $\Gamma$ is homeomorphic to one of the graphs of Figure \ref{specs} the union $w_1\Gamma \cup D_1\subset X$ contains a scroll with two internal points which is impossible because of Example 3. 
If $\Gamma$ is a simple circle then the union $w_1\Gamma\cup D_1$ is either ${\cal S}_2$ 
(iff $\Gamma-\int (\alpha_1)$ is a single edge contained in $\partial D_1$), 
or the union $D_2= w_1\Gamma\cup D_1$ is a disc with 3 internal points $v, w, w_1$. Both these possibilities contradict our assumptions concerning $X$. This completes the proof. 

\end{proof}

\begin{proof}[Proof of Theorem \ref{thm1}]  
Consider a random graph $\Gamma\in G(n,p)$ and its clique complex $X_\Gamma$. Clearly, $X_\Gamma$ is connected if and ony if $\Gamma$ is connected. 
Since $p\ll n^{-1/2}$ we know that $\dim X_\Gamma\le 3$ a.a.s.; see Corollary \ref{dimension}. 
The 3-simplexes of $X_\Gamma$ are in one-to-one correspondence with the embedding of the complete graph $K_4$ into $\Gamma$. 
Let us show that each 3-simplex of 
$X_\Gamma$ has at least three free faces. Indeed, assume that there is a 3-simplex in $X_\Gamma$ with less than three free faces. Then the complex $S$ formed as the union 
$S=S_1\cup S_2\cup S_3$ of three tetrahedra $S_1, S_2, S_3$, where the intersections $S_1\cap S_2$ and $S_1\cap S_3$ are 2-simplexes and $S_2\cap S_3$ is an edge,
would be embeddable into $X_\Gamma$; however this is impossible due to Theorem \ref{embed} since $\nu(S)=6/12 = 1/2$ and $p\ll n^{-1/2}$. 
Choosing a free face in each 3-simplex and performing collapse $X_\Gamma\searrow X'_\Gamma$ we obtain a 2-complex $X'_\Gamma$. 
Clearly, $X'_\Gamma$ does not contain ${\cal S}_1$ and ${\cal S}_2$ as subcomplexes.

%

Next we perform a sequence of simplicial collapses $$X'_\Gamma\searrow X''_\Gamma\searrow X'''_\Gamma\searrow\dots$$ where on each step we collapse all free faces of 2-simplexes. After finitely many such collapses we obtain a complex $X_\Gamma^\infty$ which is either (a) a graph, or (b) a closed 2-dimensional simplicial complex. We know that 
$X_\Gamma^\infty$ contains neither ${\cal S}_1$ nor ${\cal S}_2$ as a subcomplex, and besides, $\pi_1(X_\Gamma, x_0) = \pi_1(X_\Gamma^\infty, x_0)$ for any base point $x_0$. 

We claim that option (b) happens with probability tending to zero as $n\to \infty$; in other words, $X_\Gamma^\infty$ is a graph, a.a.s. Indeed, by Theorem \ref{thm2} if $X^\infty_\Gamma$ is not a graph then it admits a simplicial embedding $S \to X_\Gamma^\infty$ of some $S\in \cal L$. However 
for a fixed $S\in \cal L$ one has 
$$\PP(S\subset X_\Gamma^\infty) \le \PP(S\subset X_\Gamma)\le n^{v(S)}p^{e(S)}\le  \left(n^{1/2}p\right)^{e(S)}$$
and 
therefore (using Theorem \ref{thm2}) the probability that $X_\Gamma^\infty$ is not a graph is less than or equal to
\begin{eqnarray*}\sum_{S\in {\cal L}} \PP(S\subset X_\Gamma^\infty) &\le& \sum_{S\in {\cal L}} \left(n^{1/2}p\right)^{e(S)} \\
&=& \sum_{E\ge 1} |{\cal L}_E| \left(n^{1/2}p\right)^E \le A\sum_{E\ge 1} \left(bn^{1/2}p\right)^E \to 0\end{eqnarray*}
as $n\to \infty$ since we assume that $pn^{1/2}\to 0$.  

\end{proof}

\section{Uniform hyperbolicity}

Let $X$ be a finite simplicial complex. For a simplicial loop $\gamma: S^1\to X^{(1)}\subset X$ we denote by $|\gamma|$ the length of $\gamma$. If $\gamma$ is null-homotopic,
$\gamma \sim 1$, we denote by $A_X(\gamma)$ {\it the area} of $\gamma$, i.e. the minimal number of triangles in any simplicial filling $V$ for $\gamma$. 
A simplicial filling (or a simplicial Van Kampen diagram) for a loop $\gamma$ is defined as a pair of simplicial maps $S^1\stackrel{i}\to V\stackrel{b}\to X$ 
such that $\gamma=b\circ i$ and the mapping cylinder of $i$ is a disc with boundary $S^1\times 0$, see \cite{BHK}.   

Clearly $I(X)=I(X^{(2)})$ i.e. the isoperimetric constant $I(X)$ depends only on the 2-skeleton $X^{(2)}$. 


Define the following invariant of $X$
$$I(X)= \inf \left\{\frac{|\gamma|}{A_X(\gamma)}; \quad \gamma: S^1\to X^{(1)}, \gamma\sim 1 \quad \mbox{in $X$}\right\}\, \in \, \R.$$

The inequality $I(X)\ge a$ means that for any null-homotopic loop $\gamma$ in $X$ one has the isoperimetric inequality $A_X(\gamma)\le a^{-1}\cdot |\gamma|$. 
The inequality $I(X) < a$ means that there exists a null-homotopic loop $\gamma$ in $X$ with $A_X(\gamma) >a^{-1}\cdot |\gamma|$, i.e. $\gamma$ is null-homotopic but does not bound a disk of area less than $a^{-1}\cdot |\gamma|$. 

%
%

It is well known that $I(X)>0$ if and only if $\pi_1(X)$ is {\it hyperbolic} in the sense of M. Gromov \cite{Gromov87}.

\begin{example}\label{extorus} {\rm For $X=T^2$ one has  $I(X) =0$. }
\end{example}
%

It is known that the number $I(X)$ coincides with the infimum of the ratios ${|\gamma|}\cdot{A_X(\gamma)}^{-1}$ where $\gamma$ runs over all null-homotopic simplicial {\it prime} loops in $X$, i.e. such that their lifts to the universal cover $\tilde X$ of $X$ are simple. 
Note that any simplicial filling $S^1\stackrel{i}\to V\stackrel{b}\to X$ for a prime loop $\gamma: S^1\to X$ has the property that 
$V$ is a simplicial disc and $i$ is a homeomorphism $i: S^1\to \partial V$.  Hence for prime loops $\gamma$ the area $A_X(\gamma)$ coincides with the minimal number of 2-simplexes in any simplicial spanning disc for $\gamma$.

The following Theorem \ref{hyp} 
gives a uniform isoperimetric constant for random complexes $X_\Gamma$ where $\Gamma\in G(n, p)$. 
It is a slightly stronger statement than simply 
hyperbolicity of the fundamental group of $Y$. 

\begin{theorem}\label{hyp} Suppose that for some $\epsilon>0$ the probability parameter $p$ satisfies
\begin{eqnarray}p\ll n^{-1/3-\epsilon}.\label{babsonrange}
\end{eqnarray}
Then there exists a constant $c_\epsilon>0$ depending only on $\epsilon$ such that the clique complex $X_\Gamma$ of a random graph
$\Gamma\in G(n, p)$, with probability tending to 1 as $n\to \infty$, has the following property: any subcomplex $Y\subset X_\Gamma$ 
satisfies $I(Y)\ge c_\epsilon$; in particular, for any subcomplex $Y\subset X_\Gamma$ the fundamental group $\pi_1(Y)$ is hyperbolic, a.a.s.
\end{theorem}

The proof of Theorem \ref{hyp} is given in the Appendix at the end of the paper.

\section{Topology of minimal cycles with $\tilde \nu(Z)>1/3$}

We start with the following Lemma which describes 2-complexes $S$ with $b_2(S) =0$ and $\nu(S) >1/3$. 

\begin{lemma}\label{b20} Let $S$ be a closed strongly connected pure 2-complex with $b_2(S)=0$. If $\nu(S)>1/3$ then $S$, as a simplicial complex, is either a triangulated projective plane $P^2$ or a 
simplicial quotient $P'$ of a triangulated projective plane $P^2$ where two vertices of $P^2$ and two adjacent edges are identified, i.e. $v(P')=v(P^2)-1$, $e(P')=e(P^2)-1$, and $f(P')=f(P^2)$. 
\end{lemma}
\begin{proof}
Since $\chi(S)=1-b_1(S)$, using formula (\ref{useful}), we see that the assumption $\nu(S)>1/3$ implies that 
\begin{eqnarray}\label{le0}
3(1-b_1(S))+L(S)>0.
\end{eqnarray} In particular, we have 
$$L(S) \ge - 2 \quad \mbox{and} \quad b_1(S)=0. $$ Since $S$ is closed we have $L(S) \le 0$ and therefore there are 3 possibilities: $L(S)=0, -1, -2$. 

If $L(S)=0$ then each edge has degree 2 and $S$ is a pseudo-surface. Using Corollary 2.1 from \cite{CF1} we obtain that $S$ is a genuine triangulated surface without singularities and the only surface satisfying $b_1(S)=b_2(S)=0$ is the projective plane. 

The case $L(S)=-1$ is impossible. 
Indeed, if $L(S)=-1$ then there is a single edge of degree 3 and all other edges have degree 2. The link of a vertex incident to the edge 
of degree 3 will be a graph with all vertices of degree 2 and one vertex of degree 3 which is impossible. 

Assume now that $L(S)=-2$. There are two possibilities: (a) either there are two edges of degree 3 and all other edges have degree 2, or (b) there is a single edge 
of degree 4 and all other edges have degree 2. 

The possibility (a) cannot happen. Indeed, if $e, e'$ are two edges of degree 3 and $v$ is a vertex incident to $e$ but not to $e'$ then the link of $v$ is a graph with all vertices of degree 2 and one vertex of degree 3 which is impossible. 

Consider now the case (b). Let $e$ be the edge of degree 4 and let $v, w$ be the endpoints of $e$. 
Repeating the arguments of the proof of Theorem 2.4 from \cite{CF1} (see Case C in \cite{CF1}) we see that $S$ is obtained from a pseudo-surface $S'$ by identifying 
two adjacent edges. Since $S'$ and $S$ are homotopy equivalent, we obtain that $b_1(S')=b_2(S')=0$. Using Corollary 2.1 from \cite{CF1} and classification of surfaces we see that $S'$ is homeomorphic to the projective plane; therefore $S$ is isomorphic to a simplicial complex of type $P'$ as explained above.  
\end{proof}

\begin{corollary}\label{b201} Let $S$ be a connected 2-complex with $b_2(S)=0$. 
If $\tilde \nu(S)>1/3$ then $S$ is homotopy equivalent to a wedge of circles and projective planes. 
\end{corollary}

\begin{definition}
A finite pure 2-complex $Z$ is said to be a minimal cycle if $b_2(Z)=1$ and for any proper subcomplex $Z'\subset Z$ one has $b_2(Z')=0$. 
\end{definition}

Any minimal cycle is closed and strongly connected. 

\begin{example} { \rm Let $Z$ be the union of two subcomplexes $Z=A\cup B$ where each $A$ and $B$ is a triangulated projective plane and the intersection 
$C=A\cap B$ is a circle which is not null-homotopic in both $A$ and $B$. }
\end{example}

\begin{definition}
A minimal cycle $Z$ is said to be of type A if it has no proper closed
$2$-dimensional subcomplexes. 
If $Z$ contains a proper closed $2$-dimensional subcomplex then we say that $Z$ is a minimal cycle of type B.
\end{definition}

\begin{lemma}\label{nonasph}
Let $Z$ be a minimal cycle of type A satisfying $\tilde \nu(Z) >1/3$. 
Then $Z$ is homotopy equivalent either to $S^2$ or to the wedge $S^2\vee S^1$. Moreover, for any face  $\sigma \subset Z$ the boundary $\partial \sigma$ is null-homotopic in $Z-\int(\sigma)$. 
\end{lemma}
\begin{proof} Let $\sigma\subset Z$ be an arbitrary face. 
Starting with the complex $Z-\int(\sigma)$ and collapsing subsequently faces across the free edges we shall arrive to a connected graph $\Gamma$ (due to our assumption about the absence of closed subcomplexes). Let us show that $b_1(\Gamma)\le 1$. The inequality $\nu(Z)>1/3$ is equivalent to $3\chi(Z)+L(Z) >0$ (see formula  (\ref{useful})) where $L(Z) \le 0$ (since $Z$ is closed) and
hence $\chi(Z)\ge 1$. Therefore $\chi(\Gamma) = \chi(Z)-1 \ge 0$ which implies $b_1(\Gamma) \le 1$. Hence, $\Gamma$ is either contractible or it is homotopy equivalent to the circle. In the first case, $Z$ is homotopy equivalent to $S^2$. In the second case, $Z$ is homotopy equivalent to the result of attaching a 2-cell to the circle,
$S^1\cup_f e^2$. Since $b_2(Z)=1$ we obtain that $\deg(f)=0$, and hence $Z$ is homotopy equivalent to $S^1\vee S^2$. We see that the inclusion 
$\partial \sigma\to Z-\int(\sigma)\simeq \Gamma$ 
is homotopically trivial in both cases. 
\end{proof}
%

\begin{lemma}\label{lmB1} Let $Z$ be a minimal cycle of type B such that $\tilde \nu(Z)>1/3$. Suppose that any edge $e$ of $Z$ has degree $\le 3$. Then $Z$ is isomorphic (as a simplicial complex) to 
the union $P^2\cup D^2$, where $P^2$ and $D^2$ are triangulated projective plane and the disc, $P^2\cap D^2= \partial D^2=P^1 \subset P^2$, and the 
loop $\partial D^2$ has either $3, 4$ or $5$ edges. Here $P^1\subset P^2$ denotes a simple homotopically nontrivial simplicial loop on the projective plane. 
In particular, $Z$ is homotopy equivalent to $S^2$ and for any face  $\sigma \subset P^2\subset Z$ the boundary $\partial \sigma$ 
is null-homotopic in $Z-\int(\sigma)$. 
\end{lemma}

\begin{proof}
Let $Z'$ be a strongly connected proper closed 2-dimensional subcomplex of $Z$.
Since any edge of $Z'$ has degree $\leq 3$ in $Z'$, it follows from Lemma \ref{b20} that $Z'$ is homeomorphic to $P^2$.

 Denote $Z''=\overline{(Z-Z')}$ and let $G$ be the graph $G=Z'\cap Z''$. Let $\Gamma$ be the subgraph of the 1-skeleton $Z^{(1)}$ of $Z$ formed by the edges of degree $3$ in $Z$. Clearly $G\subset \Gamma$. By definition of $\Gamma$ and the assumptions of the Lemma, any edge of $\Gamma$ has degree $3$ in $Z$ and 
every edge of $Z$ which is not in $\Gamma$ must have degree $2$ in $Z$. In particular one has that $L(Z)=-e(\Gamma)$.

The graph $G$ (and therefore $\Gamma$) must contain a cycle, since otherwise $Z$ is homotopy equivalent to $Z'\vee Z''$ and thus $b_2(Z'')=1$, contradicting the minimality of $Z$. In particular $e(G)\geq 3$. Moreover, $\Gamma$ has at most 5 edges since $\tilde\nu(Z) > 1/3$ implies $L(Z)\geq -5$ (using formula (\ref{useful}) and $L(Z)=-e(\Gamma)\leq -e(G)$). 
Hence $\Gamma$ either contains exactly one cycle (of length 3, 4 or 5) or $\Gamma$ is a square with one diagonal. 
\begin{figure}[h]
\centering
\includegraphics[width=0.19\textwidth]{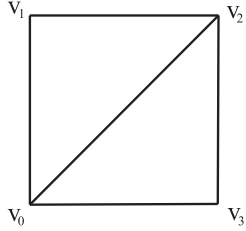}
\caption{Graph $\Gamma$.}
\end{figure}

Let us show that the latter case is impossible. Indeed, suppose that $\Gamma$ is a square with one diagonal. 
Let $v_0$ be one of the vertices of degree $3$ in $\Gamma$. Then $v_0$ is incident to exactly three odd degree edges in $Z$ 
(corresponding to the three neighbours $v_1, v_2, v_3$ of $v_0$ in $\Gamma$). In particular the link $\lk_Z(v_0)$ would have an odd number of odd degree vertices which is impossible. We conclude that $b_1(\Gamma)=b_1(G)=1$.

We now show that $\Gamma$ is a cycle and therefore $G=\Gamma$. Suppose that $\Gamma$ contains an edge $e$ with a free vertex $v$. Then the link $\lk_Z(v)$ is a graph with exactly one vertex of degree $3$ and all other vertices of degree $2$. This contradicts the fact that every graph has an even number of odd degree vertices.

We have shown that $G=\Gamma$ is a cycle of length 3, 4 or 5 and that all edges of $G$ have degree $3$ in $Z$ and all edges of $Z$ which are not in $G$ have degree $2$. Recall that $Z=Z'\cup_G Z''$ where $Z'$ is a triangulated projective plane. Since for any edge $e\in G$, one has $deg_{Z''}(e)=deg_Z(e)-deg_{Z'}(e)=1$ it follows that 
$Z''$ is a pseudo-surface with boundary. Moreover, since $\chi(G)=0$ and $\chi(Z')=1$ we obtain $2=\chi(Z)=\chi(Z')+\chi(Z'')$, i.e. $\chi(Z'')=1$. Hence $Z''$ is a disk. Besides, $G=Z'\cap Z''$ is not null-homotopic in $Z'$ since otherwise $G$ bounds a disc $A^2\subset Z'$ and $b_2(Z)=b_2(Z')+b_2(Z''\cup A)$ implying 
$b_2(Z''\cup A)=1$ which would contradict the minimality of $Z$. 
Hence we see that $Z$ is homotopy equivalent to $S^2$ and 
any 2-simplex $\sigma \subset Z'$ has the required property.

\end{proof}

\begin{lemma}\label{lmB2}
Let $Z$ be a minimal cycle of type B such that $\tilde \nu(Z)>1/3$ and such that an edge $e$ of $Z$ has degree $\ge 4$.
Then $Z$ is isomorphic (as a simplicial complex) to the quotient $q: \hat Z=P^2\cup D^2 \to Z$ of a minimal cycle $\tilde Z$ of type B with $\tilde \nu(\hat Z)>1/3$ and such that all edges of $\tilde Z$ have degree $\le 3$ (as described in the previous Lemma); the map $q$ identifies two vertices and two adjacent edges. In particular, $Z$ is homotopy equivalent to $S^2$ and for any face 
$\sigma \subset q(P^2)\subset Z$ the boundary $\partial \sigma$ 
is contractible in $Z-\int(\sigma)$. 
\end{lemma}

\begin{proof}
Let $\Gamma$ be the subgraph of the 1-skeleton of $Z$ which is the union of the edges of degree $\geq 3$. As in the proof of the previous lemma, the inequality 
$\tilde \nu(Z)>1/3$ implies $L(Z)\ge -5$ and using our 
assumption that at least one edge of $\Gamma$ has degree $\geq 4$ we obtain 
$-5\leq L(Z)\leq-e(\Gamma)-1$, i.e. $\Gamma$ has at most 4 edges.  
On the other hand $e(\Gamma)\geq 3$ since $\Gamma$ must contain a cycle as follows from the argument used in the proof of Lemma \ref{lmB1}. 
Thus we have consider the cases $e(\Gamma)$ equals 3 or 4. 


Define $\Gamma_{\rm odd}$ to be the subgraph of $\Gamma$ formed by the edges of odd degree in $Z$. The graph $\Gamma_{\rm odd}$ is non-empty; 
indeed, since $e(\Gamma)\geq 3$ and every edge of $\Gamma$ with even degree must have degree $\geq 4$, the assumption 
$\Gamma_{\rm odd}=\emptyset$ would imply $L(Z)\leq -2e(\Gamma)\leq -6$ contradicting $L(Z)\ge -5$. 
Furthermore the graph $\Gamma_{\rm odd}$ may not have a free vertex. 
If $\Gamma_{\rm odd}$ contained an edge $e$ with a free vertex $v$ then the link $\lk_Z(v)$ would be graph with exactly one vertex of odd degree contradicting the fact that every graph has an even number of odd degree vertices. We obtain in particular that $e(\Gamma_{\rm odd})\geq 3$ and $b_1(\Gamma_{\rm odd})\geq 1$.

We can now describe the graph $\Gamma$. 

If $e(\Gamma)=3$, then all edges of $\Gamma$ must have odd degree in $Z$, i.e. $\Gamma =\Gamma_{\rm odd}$. Furthermore, since $L(Z)\geq -5$ and $Z$ has at least one edge of degree $>3$, it follows that $\Gamma$ is a cycle formed by two edges of degree 3 and one edge of degree 5. In particular, $L(Z)=-5$. 
Denote the edge of degree 5 by $e$. Let $v$ be a vertex of $e$. Then the link $\lk_Z(v)$ is a graph with exactly two vertices of odd degree. One of these vertices has 
degree 3 in the link $\lk_Z(v)$ and the other vertex has degree 5. 
\begin{figure}[h]
\centering
\includegraphics[width=0.19\textwidth]{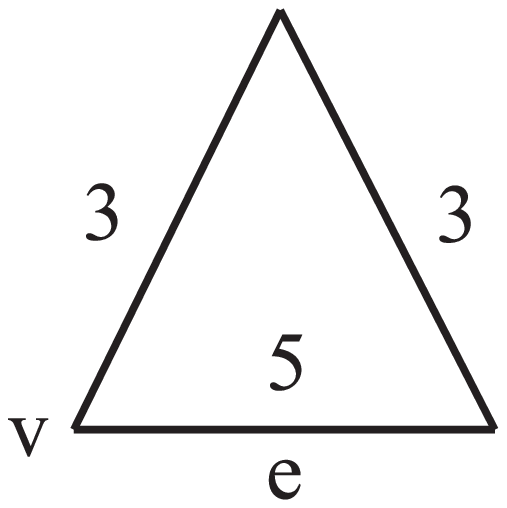}
\end{figure}
The link $\lk(v)$ is connected since otherwise we would have $b_1(Z)\ge 1$ (by Corollary 2.1 from \cite{CF1}) implying $\chi(Z)\le 1$ and $L(Z) \ge -2$, a contradiction. Hence, the link $\lk(v)$ is a connected graph with one vertex of degree 3, one vertex of degree 5 and all other vertices of degree 2. 
There are two possibilities for $\lk(v)$ which are shown in Figure \ref{links}. 
\begin{figure}[h]
\centering
\includegraphics[width=0.6\textwidth]{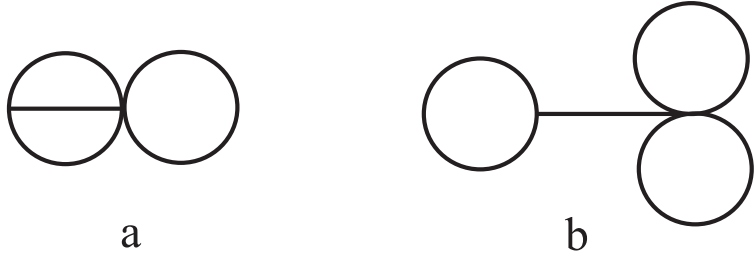}
\caption{Links of a vertex incident to an edge of degree 4.}\label{links}
\end{figure}
A neighbourhood of the point $v$ is the cone $v\cdot \lk(v)$ over the link $\lk(v)$. We may represent $\lk(v)$ as the union $A \cup B$ where $A$ is a circle and the intersection $A\cap B$ is one point, the vertex of degree 5. 
We may cut $Z$ from the vertex $v$ and along the edge $e$ introducing instead of $v$ two new vertices
$v_1$ and $v_2$ end two edges (of degree 3 and 2) instead of $e$. Formally we replace the cone $v\cdot \lk(v)$ by the union of two cones 
$(v_1\cdot A)\cup (v_2\cdot B)$ as shown in Figure \ref{cones}. 
The obtained 2-complex $\hat Z$ is a minimal cycle 
$\chi(\hat Z)=\chi(Z)=2$ and 
$L(\hat Z)= -3$. To apply Lemma \ref{lmB1} we want to show that $\tilde{\nu}(\hat{Z})>1/3$. 
The negation $\tilde \nu(\hat Z) \le 1/3$ means that there exists a subgraph $H\subset \hat{Z}^{(1)}$ with $\nu(H)\leq 1/3$. 
Identifying two adjacent edges of $H$ we obtain a subgraph $H'$ of the 1-skeleton $Z^{(1)}$ with 
$v(H')=v(H)-1$, $e(H')=e(H)-1$ and now the inequality $e(H)\ge 3 v(H)$ implies $e(H') \ge 3v(H')$, and therefore $\tilde \nu(Z) \le 1/3$ 
which contradicts our assumption
$\tilde{\nu}({Z})>1/3$.


From Lemma \ref{lmB1} we know that $\hat Z$ is isomorphic to 
$P^2\cup D^2$ where the intersection $P^2\cap D^2=P^1\subset P^2$ has length 3 (since $L(\hat Z)=-3$). Therefore, we obtain that $Z$ can be obtained from 
$P^2\cup D^2$ by identifying two adjacent edges. 
\begin{figure}[h]
\centering
\includegraphics[width=0.45\textwidth]{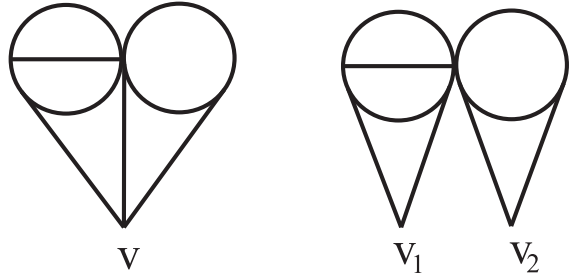}
\caption{Resolving the cone.}\label{cones}
\end{figure}

Consider now the remaining case $e(\Gamma)=4$.  Then $\Gamma$ has three edges of degree 3 and one edge of degree 4. 
Besides, the edges of degree 3 form a cycle since $b_1(\Gamma_{\rm odd})\geq 1$.
Suppose $e(\Gamma)=4$, i.e. $\Gamma=\Gamma_{\rm odd}\cup e$ where $\Gamma_{\rm odd}$ is a cycle of length 3 with all edges of degree 3, and where $e$ is the edge of degree 4. 
Then $e$ contains a free vertex $v$ in $\Gamma$. 
Since $\deg_Z(e)=4$, we see that the link $\lk_Z(v)$ is topologically a wedge of two circles. A neighbourhood of $v$ in $Z$ is a cone $v\cdot \lk(v)$ and representing 
the link $\lk(v)$ and a union of two circles $A\cup B$ intersecting at the vertex of degree 4, we may replace the cone $v\cdot \lk(v)$ by the union of two cones
$(v_1\cdot A) \cup (v_2\cdot B)$ where $v_1$ and $v_2$ are two new vertices.
We obtain a simplicial complex $\hat{Z}$ such that $Z$ is obtained from $\hat{Z}$ by identifying two adjacent edges. 
Clearly $\hat{Z}$ is a minimal cycle of type $B$ with all edges of degree $\leq 3$. 
As in the case $e(\Gamma)=3$ considered above one shows that $\tilde{\nu}(\hat{Z})>1/3$. 
Thus, we see that $\hat Z$ is a minimal cycle satisfying conditions of Lemma \ref{lmB1} and $Z$ is obtained from $\hat Z$ by identifying two adjacent edges. 
\end{proof}

\begin{corollary}\label{freeproduct} Let $X$ be a connected 2-complex satisfying $\tilde \nu(X)>1/3$. Then $X$ is homotopy equivalent to a wedge of circles, 2-spheres and real projective planes. Besides, there exists a subcomplex $X'\subset X$ containing the 1-skeleton of $X$ and having the homotopy type of a wedge of circles and real projective planes and such that $\pi_1(X') \to \pi_1(X)$ is an isomorphism. 
In particular, the fundamental group of $X$ is a free product of several copies of $\Z$ and $\Z_2$ and hence it is hyperbolic. 
\end{corollary}

This Corollary is equivalent to Theorem 1.2 from \cite{Babson}. The proof given below is independent of the arguments of \cite{Babson}. Our proof is based on 
the classification of minimal cycles described above in the this section. 
This classification of minimal cycles is not only useful for the proof of Corollary \ref{freeproduct} but it is also plays an important role in the proofs of many results presented in this paper. 

\begin{proof}
We will act by induction on $b_2(X)$. 

If $b_2(X)=0$ and $\tilde \nu(X)>1/3$ then using Corollary \ref{b201} we see that the complex $X$ is homotopy equivalent to a wedge of circles and projective planes. In this case one sets $X'=X$ and the result follows. 

Assume now that Corollary \ref{freeproduct} was proven for all connected 2-complexes $X$ satisfying $\tilde \nu(X)>1/3$ and $b_2(X)<k$. 

Consider a 2-complex $X$ satisfying $b_2(X)=k>0$ and $\tilde \nu(Z) >1/3$. 

Find a minimal cycle $Z\subset X$. Then the homomorphism $H_2(Z; \Z)=\Z \to H_2(X;\Z)$ is an injection. 
Let $\sigma\subset Z$ be a simplex given by Lemmas \ref{nonasph}, \ref{lmB1}, \ref{lmB2}. 
Then $Y= X- \int(\sigma)$ satisfies $b_2(Y)=k-1$. Indeed, 
$H_2(X,Y)=\Z$ and 
in the exact sequence
$$0\to H_2(Y) \to H_2(X) \to H_2(X, Y) \stackrel{\partial_\ast}\to H_1(Y) \to \dots$$
the homomorphism $\partial_\ast=0$ is trivial since the curve $\partial \sigma$ is contractible in $Y$. 
Since $\tilde \nu(Y) >1/3$, by the induction hypothesis there exists a subcomplex $Y'\subset Y$ such that $\pi_1(Y') \to \pi_1(Y)$ is an isomorphism and $Y'$ is homotopy equivalent to a wedge of circles and projective planes. However $X$ is homotopy equivalent to $Y\vee S^2$ and the result follows (with $X'=Y'$).

%
%
%
%
%
%

\end{proof}

\section{The Whitehead Conjecture}

If $p\ll n^{-1/3}$ then $\dim X_\Gamma\le 5$ a.a.s. (see Corollary \ref{dimension}). We consider below the 2-dimensional skeleton $X_\Gamma^{(2)}$ which can be viewed as a random 2-complex. In this section we shall examine the validity of the Whitehead Conjecture for aspherical subcomplexes of $X_\Gamma^{(2)}$.

Recall that for any simplicial complex $K$, its first barycentric subdivision $K'$ is a clique complex. 
Thus, if there exists a counterexample to the Whitehead Conjecture then there exists a counterexample of the form $X^{(2)}_\Gamma$ for certain graph $\Gamma$.

\begin{theorem}\label{aspherical} Assume that $p\ll n^{-1/3-\epsilon}$, where $\epsilon >0$ is fixed. Then, for a random graph $\Gamma\in G(n,p)$ the 2-skeleton $X^{(2)}_\Gamma$ 
of the clique complex $X_\Gamma$ has the following property with probability tending to 1 as $n\to \infty$: a subcomplex $Y\subset X^{(2)}_\Gamma$ is aspherical if and only if every subcomplex $S\subset Y$ having at most $2\epsilon^{-1}$ edges is aspherical. 
\end{theorem}

Intuitively, this statement asserts that a subcomplex $Y\subset X_\Gamma^{(2)}$ is aspherical iff it has no \lq\lq small bubbles\rq\rq\  where by a \lq\lq bubble\rq\rq \ we understand 
a subcomplex $S\subset Y$ with $\pi_2(S)\ne 0$ and \lq\lq a bubble is small\rq\rq\  if it satisfies the condition $e(S)\le 2\epsilon^{-1}$.

\begin{corollary} \label{wc} Assume that $p\ll n^{-1/3-\epsilon}$, where $\epsilon >0$  is fixed. Then, for a random graph $\Gamma\in G(n,p)$, the clique complex $X_\Gamma$ has the following property with probability tending to 1 as $n\to \infty$: any aspherical subcomplex $Y\subset X_\Gamma^{(2)}$ satisfies the Whitehead Conjecture, i.e. any subcomplex $Y'\subset Y$ is also aspherical. 
\end{corollary}


Here is another interesting statement about the local structure of aspherical subcomplexes of $X_\Gamma^{(2)}$. 

\begin{corollary}\label{aspherical1} Assume that $p\ll n^{-1/3-\epsilon}$, where $\epsilon >0$ is fixed. 
Then, for a random graph $\Gamma\in G(n,p)$ the clique complex $X_\Gamma$ 
has the following property with probability tending to 1 as $n\to \infty$: 
for any aspherical subcomplex 
$Y\subset X^{(2)}_\Gamma$ any subcomplex $S\subset Y$ 
with $e(S) \le 2\epsilon^{-1}$ is collapsible to a graph. 
\end{corollary}

We now start preparations for the proofs of theorems \ref{aspherical} and \ref{aspherical1} which appear below in this section.
Corollary \ref{wc} obviously follows from 
Theorem \ref{aspherical}.

Let $Y$ be a simplicial complex with $\pi_2(Y)\not=0$. As in \cite{CF1}, we define a numerical invariant $M(Y)\in \Z$, $M(Y) \ge 4$, as the minimal number of faces in a 
2-complex $\Sigma$ 
homeomorphic to the sphere
$S^2$ such that there exists a homotopically nontrivial simplicial map $\Sigma\to Y$. 

We define $M(Y)=0$, if $\pi_2(Y)=0$.

\begin{lemma}[See Corollary 5.3 in \cite{CF1}]\label{lmm} Let $Y$ be a 2-complex with $I(Y)\ge c>0$. Then 
$$M(Y)\le \left(\frac{16}{c}\right)^2.$$
\end{lemma}
Combining this with Theorem \ref{hyp} we obtain:

\begin{lemma}\label{lm17}
Assume that the probability parameter satisfies $p\ll n^{-1/3-\epsilon}$ where $\epsilon>0$ is fixed. Then there exists a constant $C_\epsilon>0$ such that for a random graph $\Gamma\in G(n,p)$ the clique complex $X_\Gamma$ has the following property with probability tending to one: for any subcomplex 
$Y\subset X^{(2)}_\Gamma$ one has $M(Y) \le C_\epsilon$. 
\end{lemma}

Clearly, Lemma \ref{lm17} follows from Theorem \ref{hyp} and from Lemma \ref{lmm}. 
%
%
%
%
%
%
%
%

\begin{proof}[Proof of Theorem \ref{aspherical}]
Let $\Gamma$ be a random graph, $\Gamma\in G(n,p)$, $p\ll n^{-1/3-\epsilon}$, and let $Y\subset X_\Gamma^{(2)}$ be a 2-dimensional subcomplex. Suppose that $\pi_2(Y)\not=0$, i.e. 
$Y$ is not aspherical. Using Lemma \ref{lm17} we have $M(Y)\le C_\epsilon$ a.a.s. where $C_\epsilon >0$ is a constant depending on $\epsilon$.  
There is a homotopically nontrivial simplicial map $\phi: S\to Y$ where $S$ is a triangulation of the sphere $S^2$ having at most $C_\epsilon$ faces. 
Hence, $Y$ must contain as a subcomplex a simplicial quotient $S'=\phi(S)$ of a triangulation $S$ of the sphere $S^2$ having at most $C_\epsilon$ faces and such that 
$\phi_\ast: \pi_2(S) \to \pi_2(S')$ is nonzero. 

Consider the set of isomorphism types $\LL=\{S'\}$ of pure 2-complexes $S'$ having at most $C_\epsilon$ faces and such that $\pi_2(S')\not= 0$; clearly the list $\LL$ is finite. 
By Theorem \ref{embed} any  $S'\in \LL$ satisfying $\tilde\nu(S')< 1/3+\epsilon$ is not embeddable into $Y$, a.a.s.
Next we show that any $S'\in \LL$ with $\tilde \nu(S') \ge 1/3+\epsilon$ contains a small bubble, i.e. a non-aspherical subcomplex with at most $2\epsilon^{-1}$ edges. 
If $b_2(S')=0$ then by Lemma \ref{b20} we see that $S'$ contains a subcomplex $S''\subset S$ which is either a triangulation of $P^2$ or a triangulation of $P^2$ with 2 adjacent edges identified. In both cases one has $\pi_2(S'') \not=0$. 
Now we use the inequality $$\nu(S'') \ge 1/3 +\epsilon$$ to show that $e(S'') \le \epsilon^{-1}$. Indeed, in the first case one has 
$$\nu(S'') = 1/3+\frac{1}{e(S')}$$
implying $e(S'')\le \epsilon^{-1}$ and in the second case 
$$\nu(S'') = 1/3+\frac{1}{3e(S')}$$
implying $e(S'') \le (3\epsilon)^{-1}\le \epsilon^{-1}.$

Consider now that case when $b_2(S')>0$. Then $S'$ contains a minimal cycle $S''\subset S'$. By Lemma \ref{nonasph} $\pi_2(S'')\not=0$ and we need to show that
$e(S'') \le 2\epsilon^{-1}$. Indeed, we know that $\nu(S'') \ge 1/3+\epsilon$ and $\chi(S'')\le 2$. Hence
$$\frac{1}{3} + \frac{2}{e(S'')} \ge \nu(S'') = \frac{1}{3} + \frac{3\chi(S'')+L(S'')}{3e(S'')}\ge \frac{1}{3} +\epsilon$$
implying $e(S'') \le 2\epsilon^{-1}$. 

Let us now prove the inverse implication, i.e. that the random complex $X_\Gamma$ with probability tending to 1 as $n\to \infty$ has the following property:
if a subcomplex $Y\subset X_\Gamma^{(2)}$ contains a small bubble $S\subset Y$, $\pi_2(S)\not=0$, 
$e(S)\le 2\epsilon^{-1}$, then $Y$ is not aspherical. There are finitely isomorphism types of 2-complexes $S$ with at most $2\epsilon^{-1}$ edges. Therefore, by Theorem 
\ref{embed} we may conclude that a random complex $X_\Gamma^{(2)}$ may contain as a subcomplex only the bubbles $S$, $\pi_2(S)\not=0$, $e(S)\le 2\epsilon^{-1}$satisfying $\tilde \nu(S)\ge 1/3+\epsilon$. 

If $b_2(S)>0$ then there is a minimal cycle $S'\subset S$, $\tilde\nu (S')\ge 1/3+\epsilon$. 
By Lemma \ref{nonasph} the Hurewicz map $h: \pi_2(S') \to H_2(S')$ is an epimorphism. Since $H_2(S') \to H_2(Y)$ is injective, we see that $H_2(Y)$ contains a spherical  homology class and hence $\pi_2(Y)\not=0$. 

If $b_2(S)=0$ then by Lemma \ref{b20} there is a subcomplex $K\subset S$ which is homotopy equivalent to the real projective plane $P^2$. By a Theorem of Crockfort
\cite{CC}, see also \cite{Adams}, an aspherical complex cannot contain such $K$ as a subcomplex; hence $\pi_2(Y)\not=0$, i.e. $Y$ is not aspherical. 

\end{proof}

\begin{proof}[Proof of Corollary \ref{aspherical1}] Let $X_\Gamma$ be the clique complex of a random graph $\Gamma\in G(n,p)$ where $p\ll n^{-1/3-\epsilon}$. 
By Theorem \ref{aspherical}, for any aspherical subcomplex $Y\subset X_\Gamma^{(2)}$, any subcomplex $S\subset Y$ with $e(S)\le 2\epsilon^{-1}$ is aspherical, a.a.s.
We shall also assume (using Theorem \ref{embed} and the finiteness of the set of isomorphism types of 2-complexes satisfying $e(S)\le 2\epsilon^{-1}$) that any subcomplex 
$S\subset Y\subset X_\Gamma^{(2)}$ has the property $\tilde \nu(S)>1/3$. 

We want to show that each $S\subset Y\subset X_\Gamma^{(2)}$, $e(S)\le 2\epsilon^{-1}$ is collapsible to a graph. Indeed, performing all possible simplicial collapses on $S$ we either obtain a graph or a closed 2-dimensional complex $S'$ with $e(S')\le 2\epsilon^{-1}$ and $\tilde \nu(S')>1/3$. If $b_2(S')>0$ then $S'$ contains a minimal cycle $Z\subset S'$, $\tilde \nu(Z)>1/3$ and using Lemma \ref{nonasph} we see that $S'$ is not aspherical - a contradiction. If $b_2(S')=0$ then by Lemma \ref{b20} we see that $S'$ contains a subcomplex $X\subset S'$ homotopy equivalent to $P^2$ and $S'$ is not aspherical by a theorem of Cockcroft \cite{CC}. Hence the only possibility is that $S$ is collapsible to a graph. 
\end{proof}

\section{2-torsion in fundamental groups of random clique complexes}

\begin{theorem}\label{cd2} Assume that 
\begin{eqnarray}
p \ll n^{-11/30}. 
\end{eqnarray}
Then the fundamental group $\pi_1(X_\Gamma)$ of the clique complex of a random graph $\Gamma\in G(n,p)$ has 
geometric dimension and cohomological dimension at most $2$, and in particular
$\pi_1(X_\Gamma)$ is torsion free, a.a.s.
Moreover, if $$n^{-1/2}\ll p\ll n^{-11/30}$$
then the geometric dimension and the cohomological dimension of $\pi_1(X_\Gamma)$ equal two.  
\end{theorem}

\begin{theorem}\label{2torsion}
Assume that 
\begin{eqnarray}
n^{-11/30}\ll p \ll n^{-1/3-\epsilon}
\end{eqnarray}
where $0<\epsilon < 1/30$ is fixed. Then the fundamental group $\pi_1(X_\Gamma)$ has 2-torsion and its cohomological dimension is infinite, a.a.s.
\end{theorem}

\begin{proof}[Proof of Theorem \ref{cd2}] Consider the set ${\mathcal C}_{60}$ of isomorphism types of simplicial complexes having at most 
$$60+\frac{3}{2}(3c^{-1}_\epsilon-1)$$
edges, where $c_\epsilon>0$ is the constant given by Theorem \ref{hyp} for $\epsilon= 1/30$. 
This set is clearly finite. 
For any $n$, consider the set $\mathcal X_n$ of graphs $\Gamma\in G(n,p)$ such that the corresponding clique complex $X_\Gamma$ does not contain as a subcomplex complexes
$S\in {\mathcal C}_{60}$ satisfying $\nu(S) \le 11/30$ and such that for any subcomplex $Y\subset X_\Gamma$ one has $I(Y)\ge c_\epsilon$. 
From Theorem \ref{embed} and Theorem \ref{hyp} we know that, under our assumption $p\ll n^{-11/30}$, the probability of this set 
$\mathcal X_n$ of graphs tends to one as $n\to \infty$. 

To prove the first part of Theorem \ref{cd2} we shall construct, for any $\Gamma\in \mathcal X_n$, a subcomplex $Y_\Gamma\subset X_\Gamma^{(2)}$ which is aspherical 
$\pi_2(Y_\Gamma)=0$ and has the same fundamental group, $\pi_1(Y_\Gamma) = \pi_1(X_\Gamma)$. The existence of such $Y_\Gamma$ implies that 
$$\gd(\pi_1(X_\Gamma)) =\cd(\pi_1(X_\Gamma))\le 2.$$ 
Here we use the results of Eilenberg and Ganea \cite{EG} in conjunction with the theorem of Swan \cite{Sw} stating that a group of cohomological dimension one is a free group. 

The equality $\cd(\pi_1(X_\Gamma))=2$ under the assumptions $n^{-1/2}\ll p\ll n^{-11/30}$ follows from the result of \cite{Kahle3}, Theorem 1.2 which states that for 
$$p\ge \left(\frac{(3/2+\epsilon)\log n}{n}\right)^{1/2}$$ 
the fundamental group $\pi_1(X_\Gamma)$ has property T (a.a.s.) implying $\cd(\pi_1(X_\Gamma))>1$.

Consider the minimal cycles $Z\in \mathcal C_{60}$ and their all possible embeddings $Z\subset X_\Gamma$ where $\Gamma\in \mathcal X_n$. 
By Lemma \ref{nonasph} each such 
$Z$ contains a 2-simplex $\sigma$ such that $\partial \sigma$ is null-homotopic in $Z-\int(\sigma)$. We remove subsequently one such 2-simplex from each of the 
minimal cycles $Z\subset X_\Gamma$. The union of the 1-skeleton of $X_\Gamma$ and the remaining 2-simplexes is a 2-complex which we denote by $Y_\Gamma$. 
Clearly $\pi_1(Y_\Gamma)=\pi_1(X_\Gamma)$. To show that $Y_\Gamma$ is aspherical we shall apply Theorem \ref{aspherical}. We need to show that any subcomplex 
$S\subset Y_\Gamma$, where $S\in \mathcal C_{60}$, is aspherical. 
By the above construction we know that $S\subset Y_\Gamma$ cannot contain minimal cycles, and therefore $b_2(S)=0$. 
Without loss of generality we may assume that $S$ is closed, pure and strongly connected; then Lemma \ref{b20} implies that $S$ must contain a triangulation of the projective plane or its quotient 
with two adjacent edges are identified. 

We know that for any triangulation $S$ of $P^2$ one has
$$\tilde \nu(S) =\nu(S) = 1/3+ 1/e(S).$$
We obtain that only triangulations $S$ of $P^2$ having less than $30$ edges, $e(S)<30$, are embeddable into $X_\Gamma$ where $\Gamma\in \mathcal X_n$. 

Recall that a triangulation of a 2-complex is called {\it clean} if for any clique of three vertices $\{v_0, v_1, v_2\}$ the complex contains also the simplex $(v_0v_1v_2)$. 
We shall use the following fact: any clean triangulation of the projective plane $P^2$ contains at least $11$ vertices and $30$ edges, see \cite{HR}. The minimal clean triangulation is shown in Figure \ref{cleanp2}; the antipodal points of the circle must be identified. 
\begin{figure}[h]
\centering
\includegraphics[width=0.4\textwidth]{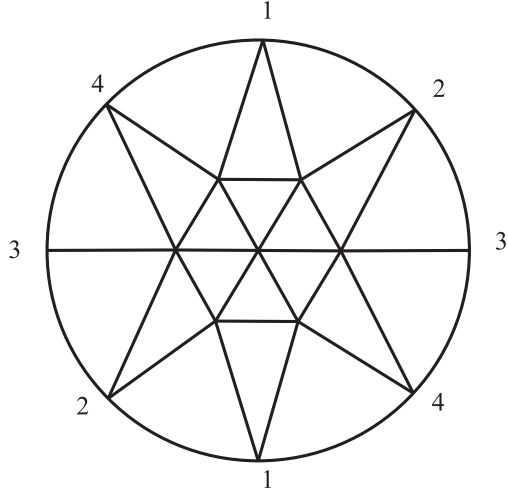}
\caption{The minimal clean triangulation of $P^2$, according to \cite{HR}.}\label{cleanp2}
\end{figure}
Any triangulation $S$ of $P^2$ containing less than $30$ edges is not clean, i.e. it contains a cycle of length 3 which is not filled by a triangle. If this cycle is null-homologous than we may 
split $S$ into two smaller surfaces one of which is a disk and another is a projective plane with smaller number of edges. Continuing by induction, we obtain that for any triangulation $S$ of $P^2$ containing less than $30$ edges there is a cycle of length 3 representing a non-contractible loop in $S$. 

We claim that $Y_\Gamma$ contains no subcomplexes $S$ with $e(S)\le 60$ which are triangulations of $P^2$. Indeed, if $S$ is embedded into $Y_\Gamma$, where $\Gamma\in \mathcal X_n$, then a nontrivial cycle of $S$ bounds a triangle in 
$X_\Gamma$. In particular, the inclusion $S\to X_\Gamma$ induces a trivial homomorphism of 
the fundamental groups $\pi_1(S) \to \pi_1(X_\Gamma)$. Since the inclusion induces an isomorphism $\pi_1(Y_\Gamma) \to \pi_1(X_\Gamma)$ we obtain that 
the inclusion $S\subset Y_\Gamma$ also induces a trivial homomorphism $\pi_1(S) \to \pi_1(Y_\Gamma)$ however now the length 3 cycle of $S$ may bound a larger disc 
and not a simple 2-simplex. We may apply Theorem \ref{hyp} about uniform hyperbolicity to estimate the size of the minimal bounding disc for this cycle. Since 
$I(Y_\Gamma) \ge c_\epsilon$ where $\epsilon = 1/30$, we see that the area of the bounding disc is $\le 3c_\epsilon^{-1}$. We obtain that there exists a subcomplex 
$S\subset L\subset Y_\Gamma$ such that $\pi_1(S) \to \pi_1(L)$ is trivial and 
\begin{eqnarray}\label{bound}e(L) \le 60+ \frac{3}{2}(3c_\epsilon^{-1} -1).\end{eqnarray}
Since $\Gamma\in \mathcal X_n$ we see that $\tilde \nu(L)>1/3$. By construction, $L$ (as well as $Y_\Gamma$) may not contain minimal cycles
since any minimal cycle $Z$ satisfying $\tilde \nu(Z)>11/30$ must have at most $60$ edges; therefore $b_2(L)=0$. 
We may assume that $L$ is strongly connected and pure. 
Then by Lemma \ref{b20} we see that each strongly connected pure component of $L$ must be isomorphic either to the projective plane or to its quotient, and in both cases we obtain a contradiction to the homomorphism $\pi_1(S) \to \pi_1(L)$ being trivial.

Similarly, one shows that $Y_\Gamma$ contains no subcomplexes $S'$ isomorphic to the quotients of a triangulation of $P^2$ with two adjacent edges identified and with
$e(S') \le 60$. One has 
$$\tilde \nu(S')=\nu(S') = \frac{1}{3} + \frac{1}{3e(S')}$$
and for $\Gamma\in \mathcal X_n$ we shall find subcomplexes $S'\subset X_\Gamma$ only if $e(S') <10$. 
Thus, using the result of \cite{HR}, we obtain that that if $S'$ is embedded into $X_\Gamma$, where $\Gamma\in \mathcal X_n$, then 
there is a cycle of length 3 in $S'$ which is not null-homotopic in $S'$; this cycle bounds a triangle in $X_\Gamma$ and as a result the inclusion $S'\to X_\Gamma$ induces a trivial homomorphism of 
the fundamental groups $\pi_1(S') \to \pi_1(X_\Gamma)$. Repeating the arguments of the preceding paragraph we find a subcomplex $S'\subset L \subset Y_\Gamma$ such that $\pi_1(S') \to \pi_1(L)$ is trivial and $L$ satisfies (\ref{bound}). As above we find that $\tilde \nu(L) >1/3$, $b_2(L)=0$ and therefore $L$ is an iterated wedge
of projective planes or projective planes with two adjacent edges identified; this contradicts the fact that $\pi_1(S')\to \pi_1(L)$ is trivial. 

%
%

\end{proof}

\subsection{Proof of Theorem \ref{2torsion}}

\subsubsection{The number of combinatorial embeddings}

Consider two 2-complexes $S_1\supset S_2$. Denote by $v_i$ and $e_i$ the numbers of vertices and faces of $S_i$. We have $v_1\ge v_2$ and $e_1\ge e_2$.
We will assume that $e_1>e_2$. 

Let $\nu(S_1, S_2)$ denote the ratio $$\nu(S_1, S_2)= \frac{v_1-v_2}{e_1-e_2}.$$
Clearly, $\nu(S_1, S_2)$ depends only on the 1-skeleta of $S_1$ and $S_2$, however it will be convenient to think of this quantity as 
being a function of the 2-complexes $S_1, S_2$. 
If $\nu(S_1)<\nu(S_2)$ then 
\begin{eqnarray}\label{one}\nu(S_1, S_2) <\nu(S_1)<\nu(S_2).\end{eqnarray}
If $\nu(S_1) >\nu(S_2)$ then 
\begin{eqnarray}\label{two}\nu(S_1, S_2) >\nu(S_1)>\nu(S_2).\end{eqnarray}
These two observations can be summarised by saying that {\it $\nu(S_1)$ always lies in the interval connecting $\nu(S_2)$ and $\nu(S_1, S_2)$. }
One has the following formula
\begin{eqnarray}\label{24}
\nu(S_1, S_2) = \frac{1}{3} +\frac{3(\chi(S_1)-\chi(S_2)) + L(S_1)-L(S_2)}{3(e_1-e_2)},
\end{eqnarray}
which follows from the equation $3v_i = e_i + 3\chi(S_i) +L(S_i)$; the latter is equivalent to (\ref{useful}). 

\begin{lemma}\label{less} Let $S_1$ be closed, i.e. $\partial S_1=\emptyset$, and $S_2$ be a pseudo-surface such that
$\chi(S_1) \le \chi(S_2)$. Then $\nu(S_1, S_2) <1/3$. 
\end{lemma}
\begin{proof} Since $S_1$ is closed, $L(S_1)\le 0$. Besides, $L(S_2)=0$ since $S_2$ is a pseudo-surface. The result now follows from the above formula. 
\end{proof}

\begin{theorem}\label{thm11} Let $S_1\supset S_2$ be two fixed 2-complexes 
and\footnote{The assumption (\ref{between}) is meaningful iff $\nu(S_1, S_2) < \tilde \nu(S_2) \le \nu(S_2)$ which, as follows from (\ref{one}) and (\ref{two}), 
implies that $\nu(S_1)<\nu(S_2)$. Thus, if Theorem \ref{thm11} is applicable, then $\nu(S_1)<\nu(S_2)$. }
\begin{eqnarray}\label{between}
n^{-\tilde \nu(S_2)} \ll p \ll n^{-\nu(S_1, S_2)}.
\end{eqnarray}
Then the number of embeddings of $S_1$ into the clique complex $X_\Gamma$ of a random graph $\Gamma\in G(n,p)$ is smaller than the number of embeddings of 
$S_2$ into $X_\Gamma$, a.a.s. 
In particular, under the assumptions (\ref{between}), with probability tending to one, there exists an embedding $S_2\to X_\Gamma$ which does not extend to an embedding $S_1\to X_\Gamma$. 
\end{theorem}
\begin{proof}
Let $T_i: G(n, p)\to \Z$ be the random variable counting the number of embeddings of $S_i$ into $X_\Gamma$, $i=1,2$ (where by an embedding we understand a simplicial injective map $S_i\to X_\Gamma$). We know that 
$$\E(T_i)=\binom n {v_i} v_i! p^{e_i} \sim n^{v_i}p^{e_i}.$$ 
Our goal is to show that $T_1<T_2$, a.a.s. We have 
$$\frac{\E(T_1)}{\E(T_2)} \sim n^{v_1-v_2}p^{e_1-e_2} = \left[ n^{\nu(S_1,S_2)}p\right]^{e_1-e_2}\to 0$$
tends to zero, under our assumption (\ref{between}) (right). 

Find $t_1, t_2 >0$ such that 
$$t_1+t_2=\E(T_2)-\E(T_1)$$ and $\E(T_1)/t_1\to 0$ while $\E(T_2)/t_2$ is bounded. 
Then 
$$P(T_1<T_2) \ge 1- P(T_1> \E(T_1) +t_1) - P(T_2< \E(T_2) -t_2).$$
By Markov's inequality 
$$P(T_1>\E(T_1)+t_1) < \frac{\E(T_1)}{\E(T_1)+t_1}= \frac{\frac{\E(T_1)}{t_1}}{1+ \frac{\E(T_1)}{t_1}}\to 0$$
while by Chebyschev's inequality
$$P(T_2< \E(T_2)-t_2) < \frac{{\rm {Var(T_2)}}}{t_2^2}.$$
It is known (see \cite{CCFK}, proof of Theorem 15) that under our assumptions (\ref{between}) the ratio 
$\frac{{\rm {Var(T_2)}}}{\E(T_2)^2}$ tends to zero.

We take $t_1=\sqrt{\E(T_1)\E(T_2)}$ and $t_2= \E(T_2) - \E(T_1) - t_1$. Then 
$$\frac{\E(T_1)}{t_1} = \sqrt{\frac{\E(T_1)}{\E(T_2)}} \to 0$$
and 
$$\frac{\E(T_2)}{t_2} = \frac{1}{1-\frac{\E(T_1)}{\E(T_2)}- \sqrt{\frac{\E(T_1)}{\E(T_2)}}}\to 1$$
is bounded. 
\end{proof}

\begin{theorem}\label{family} Let $$S_{ j}\supset S, \quad j=1, \dots, N,$$ be a finite family of 2-complexes containing a 
given 2-complex $S$ and satisfying  $\nu(S_{ j})<\nu(S)$. Assume that 
\begin{eqnarray}\label{between2}
n^{-\tilde \nu(S)} \ll p \ll n^{-\nu(S_{j}, S)}, \quad \mbox{for any} \quad j=1, \dots, N.
\end{eqnarray}
Then,
with probability tending to one, for the clique complex $X_\Gamma$ of a random graph $\Gamma\in G(n, p)$ there exists an embedding $S\to X_\Gamma$ which does not extend to an embedding $S_{j}\to X_\Gamma$, for any $j=1, \dots, N$. 
\end{theorem}
\begin{proof} Let $T_{1, j}: G(n, p) \to \Z$ denote the random variable counting the number of embeddings of $S_j$ into $X_\Gamma$. Denote 
$T_1= \sum_{j=1}^N T_{1, j}.$ Besides, Let $T_2: G(n, p)\to \Z$ denote the number of embeddings of $S$ into a random clique complex $X_\Gamma$. 
One has 
$$\frac{\E(T_1)}{\E(T_2)} = \sum_{j=1}^N \frac{\E(T_{1, j})}{\E(T_2)} \to 0$$
thanks to our assumption (\ref{between2}) (right). Taking 
$t_1= \sqrt{\E(T_1)\E(T_2)}$ and $t_2=\E(T_2)-\E(T_1)-t_1$ (as in the proof of the previous theorem) one has 
$t_1+t_2=\E(T_2) - \E(T_1)$ and $\E(T_1)/t_1\to 0$ while $\E(T_2)/t_2$ is bounded. Repeating the arguments used in the proof of the previous theorem we see that 
$T_1>T_2$ a.a.s. Since every embedding $S_j\to X_\Gamma$ determines (by restriction) an embedding $S\to X_\Gamma$, the inequality $T_1(\Gamma)>T_2(\Gamma)$ implies that there there are embeddings $S\to X_\Gamma$ which admit no extensions to $S_j\to X_\Gamma$, for any $j=1, \dots, N$. 
\end{proof}

\subsection{Projective planes in clique complexes of random graphs}

Recall that a connected subcomplex $S\subset X$ is said to be {\it essential} if the induced homomorphism $\pi_1(S) \to \pi_1(X)$ is injective.

\begin{theorem}\label{ordertwo} Let $S$ be a clean triangulation of the real projective plane $P^2$ having 11 vertices, 30 edges and 20 faces. Assume that 
$0<\epsilon < 1/30$ and 
$$n^{-11/30}\ll p\ll n^{-1/3 - \epsilon}.$$ Then the clique complex $X_\Gamma$ of a random graph $\Gamma\in G(n, p)$ contains $S$ as an essential 
subcomplex, a.a.s. In particular, the fundamental group 
$\pi_1(X_\Gamma)$ has an element of order two and hence its cohomological dimension is infinite. 
\end{theorem}
\begin{proof} 
Consider the set $\mathcal S_\epsilon$ of isomorphism types of pure connected closed 2-complexes $X$ satisfying the following conditions:
\begin{enumerate}
  \item[(a)] $X$ contains $S$ as a subcomplex; 
  \item[(b)] The inclusion $S\to X$ induces a trivial homomorphism $\pi_1(S) \to \pi_1(X)$; 
  \item[(c)] For any subcomplex $S\subset X'\subset X$, $X'\not=X$, the homomorphism $\pi_1(S)\to \pi_1(X')$ is nontrivial;
  \item[(d)]  $X$ has at most $$20+4c_\epsilon^{-1}$$ faces, where $c_\epsilon$ is the constant given by Theorem \ref{hyp}. 
  \item[(e)] $\tilde \nu(X) > 1/2 +\epsilon$;
  \end{enumerate}
The set $\mathcal S_\epsilon$ is finite due to the condition (d). 

Let us show that for any $X\in {\mathcal S}_\epsilon$ one has 
\begin{eqnarray}\label{rel}
\nu(X,S) \le 1/3,
\end{eqnarray}
which is equivalent to the inequality 
\begin{eqnarray}\label{rel1}3(\chi(X)-\chi(S))+L(X)-L(S) \le 0\end{eqnarray}
by formula (\ref{24}). 
From the exact sequence of the pair $(X,S)$
$$0\to H_2(X) \to H_2(X,S) \to H_1(S)=\Z_2 \to 0$$
we see (since the middle group has no torsion) that $b_2(X)=b_2(X,S) \ge 1$. 
Let $Z$ be a minimal cycle in $X$. Let us show that the assumption that $Z$ is of type A leads to a contradiction. Clearly, $Z \not\subset  S$ and let 
$\sigma$ be a simplex of $Z-S$. 
Then $\partial \sigma$ is null-homotopic in 
$Z-\int(\sigma)$ and hence in $X-\int(\sigma)$, and therefore $\pi_1(X-\int(\sigma))\to \pi_1(X)$ is an isomorphism and $\pi_1(S) \to\pi_1(X-\int(\sigma))$ is injective
violating (c). 
Thus $Z$ is a minimal cycle of type B, i.e. there exists a proper pure and strongly connected closed subcomplex $Z' \subset Z$. 
Since $\tilde \nu(Z') >1/3$ we see (using Lemma \ref{b20}) that
$Z'$ is a triangulated $P^2$ (or quotient its quotient with two adjacent edges identified). If $Z' \not\subset S$ then we can again remove any 2-simplex of $Z'-S$ to find a contradiction with (c), similarly to the argument given above. 
Thus $Z'\subset S$ implying that $Z'=S$. However, $\pi_1(S) \to \pi_1(Z)$ is trivial (by Lemma \ref{nonasph}) and now the property (c) gives $X=Z$. 
Therefore, we obtain $b_2(X)=1.$
This also implies that $L(X)$ is either $-3, -4$ or $-5$. 
Now we apply formula (\ref{24}) with $\chi(X)=2$, $\chi(S)=1$, $L(S)=0$ and $L(X)\le -3$ to obtain (\ref{rel1}).

Applying Theorem \ref{family} we find that for 
$n^{-11/30}\ll p \ll n^{-1/3-\epsilon}$, with probability tending to 1, there exist embedding $S\to X_\Gamma$ (where $\Gamma\in G(n,p)$ is random) 
which cannot be extended to an embedding of $X\to X_\Gamma$ for any $X\in {\mathcal S}_\epsilon$. 
Let us show that any such embedding $S\subset X_\Gamma$ induces a monomorphism $\pi_1(S) \to \pi_1(X_\Gamma)$. 
%
If $S\subset X_\Gamma$ is not essential then the central cycle $\gamma$ of $S$ (of length 4) bounds in $X_\Gamma$ a simplicial disc. 
Under the assumption $p\ll n^{-1/3-\epsilon}$, using Theorem \ref{hyp}, we find that the circle $\gamma$ bounds in $X_\Gamma$ a simplicial disk $b:D^2\to X_\Gamma$ 
of area $\le 4c_\epsilon^{-1}$ where $c_\epsilon>0$ in the constant of Theorem \ref{hyp} which depends only on the value of $\epsilon$. Consider the union $Y=S\cup b(D^2)$. This is a subcomplex of 
$X_\Gamma$ satisfying properties (a), (b), (d). We may assume that $\Gamma\in G(n,p)$ is such that any subcomplex $T\subset X_\Gamma^{(2)}$ with at most $20+3c_\epsilon^{-1}$ faces satisfies $\tilde \nu(T) >1/3+\epsilon$; the set of such graphs $\Gamma\in G(n,p)$ has probability tending to one as $n\to \infty$ according to Theorem \ref{embed}. Hence, we see that 
$Y$ satisfies the property (e) as well. 
However the property (c) can be violated. In this case we find a minimal subcomplex $S\subset X\subset Y$ which satisfies all the properties (a)-(e). 
Hence, if $S\subset X_\Gamma$ is not essential then there would exist a complex $X\in {\mathcal S}_\epsilon$ such that the embedding $S\subset X_\Gamma$ extends to an embedding $X\subset X_\Gamma$ contradicting our construction. 
\end{proof}

\section{Absence of odd torsion}
In this section we prove the following statement complementing Theorems \ref{cd2} and \ref{2torsion}.

\begin{theorem}\label{oddtorsion} Let $m\ge 3$ be a fixed prime. Assume that 
\begin{eqnarray}
 p \ll n^{-1/3-\epsilon}
\end{eqnarray}
where $\epsilon >0$ is fixed. Then a random graph $\Gamma\in G(n,p)$ with probability tending to 1 has the following property: the fundamental group of any subcomplex $Y\subset X_\Gamma$ has no $m$-torsion. 
\end{theorem}

Let $\Sigma$ be a simplicial 2-complex homeomorphic to {\it the Moore surface} 
$$M(\Z_m, 1)=S^1\cup_{f _m}e^2, \quad \mbox{where}\quad \quad m\ge 3;$$ 
it is obtained from the circle $S^1$ by attaching a 2-cell via the degree $m$ map $f_m: S^1\to S^1$, $f_m(z)=z^m$, $z\in S^1$. 
The 2-complex $\Sigma$ has a well defined circle $C\subset \Sigma$ (called {\it the singular circle}) which is the union of all edges of degree $m$; all other edges of 
$\Sigma$ have degree $2$. Clearly, the homotopy class of the singular circle generates the fundamental group $\pi_1(\Sigma)\simeq \Z_m$. 

As in \cite{CF1}, define an integer $N_m(Y)\ge 0$ associated to any connected 2-complex $Y$. If $\pi_1(Y)$ has no $m$-torsion we set $N_m(Y)=0.$
If $\pi_1(Y)$ has elements of order $m$ we consider homotopically nontrivial simplicial maps 
$\gamma: C_r \to Y$,
where $C_r$ is the simplicial circle with $r$ edges, such that
\begin{enumerate}
  \item[(a)] $\gamma^m$ is null-homotopic (as a free loop in $Y$);
  \item[(b)] $r$ is minimal: for $r'<r$ any simplicial loop $\gamma:C_{r'} \to Y$ satisfying (a) is homotopically trivial. 
  \end{enumerate} 
Any such simplicial map $\gamma:C_r \to Y$ can be extended to a simplicial map $f: \Sigma \to Y$ of a triangulation $\Sigma$ of the Moore surface, such that the singular circle $C$ of $\Sigma$ is isomorphic to $C_r$ and $f|C=\gamma$. We shall say that a simplicial map $f:\Sigma \to Y$ is {\it $m$-minimal} if 
it satisfies (a), (b) and the number of 2-simplexes in $\Sigma$ is the smallest possible. 
Now, we denote by 
$$N_m(Y)\in \Z$$ 
the number of 2-simplexes in a triangulation of the Moore surface $\Sigma$ admitting an $m$-minimal map $f: \Sigma \to Y$.


\begin{lemma}\label{lmc}
Let $Y$ be a 2-complex satisfying $I(Y)\ge c>0$. 
Let $m\ge 3$ be an odd prime. 
Then
one has 
$$N_m(Y) \le \left(\frac{6m}{c}\right)^2. $$
\end{lemma}

This is Lemma 4.7 from \cite{CF1}; we refer the reader to \cite{CF1} for a proof.

\begin{theorem}\label{thmm}
Assume that the probability parameter $p$ satisfies $p\ll n^{-1/3-\epsilon}$ where $\epsilon >0$ is fixed. 
Let $m\ge 3$ be an odd prime. 
Then there exists a constant $C_\epsilon>0$ such that a random graph $\Gamma\in G(n,p)$ with probability tending to 1 has the following property: for any subcomplex $Y\subset X_\Gamma$ one has
\begin{eqnarray}\label{ineqq}
N_m(Y) \le C_\epsilon. 
\end{eqnarray}
\end{theorem}
\begin{proof}  We know from Theorem \ref{hyp} that, with probability tending to 1, a random 2-complex $X_\Gamma$ has the following property: 
for any subcomplex $Y\subset X_\Gamma$ one has 
$I(Y)\ge c_\epsilon>0$ where $c_\epsilon>0$ is the constant given by Theorem \ref{hyp}. 
Then, setting $C= \left(\frac{6m}{c_\epsilon}\right)^2$, the inequality (\ref{ineqq}) follows from Lemma \ref{lmc}. 
\end{proof}

\begin{proof}[Proof of Theorem \ref{oddtorsion}] Let $c_\epsilon>0$ be the number given by Theorem \ref{hyp}. 
Consider the finite set of all isomorphism types of triangulations $\mathcal S_m =\{\Sigma\}$ of the Moore surface $M(\Z_m,1)$ having at most 
$\left(\frac{6m}{c_\epsilon}\right)^2$ two-dimensional simplexes. Let $\mathcal X_m$ denote 
the set of isomorphism types of images of all surjective simplicial maps $\Sigma \to X$ inducing injective homomorphisms $\pi_1(\Sigma)=\Z_m \to \pi_1(X)$,
where $\Sigma \in \mathcal S_m$. 
The set $\mathcal X_m$ is also finite.

From Theorem \ref{thmm} we obtain that, with probability tending to one, for any subcomplex $Y\subset X_\Gamma$, either $\pi_1(Y)$ has no $m$-torsion, or 
there exists an $m$-minimal map $f: \Sigma \to Y$ with $\Sigma$ having at most $\left(\frac{6m}{c_\epsilon}\right)$ simplexes of dimension 2; in the second case the image 
$X=f(\Sigma)$ is a subcomplex of $Y'$ and $f:\Sigma \to X$ induces a monomorphism $\pi_1(\Sigma)\to \pi_1(X)$, i.e. 
$X\in {\mathcal X}_m$.

From Corollary \ref{freeproduct} we know that the fundamental group of any 2-complex satisfying $\tilde \nu(X)>1/3$ is a free product of several copies of $\Z$ and $\Z_2$ and has no $m$-torsion, as we assume that $m\ge 3$. 
Since the fundamental group of any $X\in \mathcal X_m$ has $m$-torsion, where $m\ge 3$, one has $\tilde \nu(X) \le 1/3$ for any $X\in \mathcal X_m$. 
 Hence, using the finiteness of $\mathcal X_m$ and the results on the containment problem (Theorem \ref{embed}) we see that for 
$p\ll n^{-1/3-\epsilon}$
the probability that a random complex $X_\Gamma$, where $\Gamma\in G(n,p)$, contains a subcomplex isomorphic to one of the complexes
$X\in \mathcal X_m$ tends to $0$ as $n\to \infty$. 
Hence, we obtain that (a.a.s.) any subcomplex $Y\subset X_\Gamma$ does not contain $X\in \mathcal X_m$ as a subcomplex and therefore the fundamental group of 
$Y$ has no $m$-torsion.
\end{proof}

\appendix

\section{Appendix: Proof of Theorem \ref{hyp}}\label{app}

In this Appendix we give a complete and self-contained proof of Theorem \ref{hyp} which plays a key role in this paper. 
As we mentioned above, this statement is closely related to Theorem 1.1 from \cite{Babson}. 
The proof of Theorem \ref{hyp} given below is similar to the arguments of  \cite{BHK}, \cite{Babson} and \cite{CF1} and is based on two auxiliary results: 
(1) the local-to-global principle of Gromov \cite{Gromov87} and on (2) Theorem \ref{uniform} giving uniform isoperimetric constants for complexes satisfying $\tilde \nu(X)\ge 1/3+\epsilon$.

The local-to-global principle of Gromov can be stated as follows:

\begin{theorem}\label{localtoglobal} Let $X$ be a finite 2-complex and let $C>0$ be a constant such that any pure subcomplex $S\subset X$ having at most
$(44)^3\cdot C^{-2}$ two-dimensional simplexes satisfies $I(S)\ge C$. Then $I(X)\ge C\cdot 44^{-1}$.
\end{theorem}

Let $X$ be a 2-complex satisfying $\tilde \nu(X) >1/3$. Then by Corollary \ref{freeproduct} the fundamental group of $X$ is hyperbolic as it is a free product of several copies of
cyclic groups $\Z$ and $\Z_2$. Hence, $I(X)>0$. The following theorem gives a uniform lower bound for the numbers $I(X)$.

\begin{theorem}\label{uniform}
Given $\epsilon >0$ there exists a constant $C_\epsilon>0$ such that for any finite pure 2-complex $X$ with $\tilde \nu(X)\ge 1/3 + \epsilon$ 
one has $I(X) \ge C_\epsilon$.
\end{theorem}

This Theorem is equivalent to Lemma 3.6 from \cite{Babson}. The key ingredient of the proof is the classification of minimal cycles (given by Lemmas \ref{nonasph}, \ref{lmB1}, \ref{lmB2} and Corollary \ref{freeproduct}). 
We do not use webs (as in \cite{BHK}, \cite{Babson}) and operate with simplicial complexes.

\begin{proof}[Proof of Theorem \ref{hyp} using Theorem \ref{localtoglobal} and Theorem  \ref{uniform}] Let $C_\epsilon$ be the constant given by Theorem
 \ref{uniform}.
Consider the set $\mathcal S$ of isomorphism types of all pure 2-complexes having at most $44^3\cdot C_\epsilon^{-2}$ faces. In particular, all complexes in $\mathcal S$ have at most $3^{-1}\cdot 44^3\cdot C_\epsilon^{-2}$ edges. Clearly, the set $\mathcal S$ is finite. 
We may present it as the disjoint union $\mathcal S= \mathcal S_1 \sqcup \mathcal S_2$ 
where any $S\in \mathcal S_1$ satisfies $\tilde \nu(S)\ge 1/3 +\epsilon$ while for
$S\in \mathcal S_2$ one has $\tilde \nu(S) < 1/3 +\epsilon$. By Theorem \ref{embed}, 
a random complex $X_\Gamma$ contain as subcomplexes of $X_\Gamma^{(2)}$ complexes  $S\in \mathcal S_2$ with probability tending to zero as $n\to \infty$. Hence, $X_\Gamma$ may contain as subcomplexes of $X_\Gamma^{(2)}$ only complexes $S\in \mathcal S_1$, a.a.s.
By Theorem \ref{uniform}, any $S\in \mathcal S_1$ satisfies $I(S)\ge C_\epsilon$. Hence we see that with probability tending to one, any subcomplex $S$ of $Y$ having at most $44^3\cdot C_\epsilon^{-2}$ faces satisfies $I(S)\ge C_\epsilon$.
Now applying Theorem \ref{localtoglobal} we obtain $I(Y')\ge
C_\epsilon\cdot 44^{-1}=c_\epsilon$, for any subcomplex $Y'\subset Y$, a.a.s.
\end{proof}

\subsection*{Proof of Theorem \ref{uniform}}\label{secthm10}

\begin{definition} {\rm  \cite{CF1}}
We will say that a finite 2-complex $X$ is tight if for any proper subcomplex $X'\subset X$, $X'\not= X$, one has
$I(X') > I(X).$
\end{definition}
Clearly, one has
\begin{eqnarray}\label{ineq}
I(X) \ge \min \{I(Y)\} 
\end{eqnarray}
where $Y\subset X$ is a proper tight subcomplex. Since $\tilde \nu(Y) \ge \tilde \nu(X)$ for $Y\subset X$, it is obvious from (\ref{ineq}) that it is enough to prove Theorem \ref{uniform} under the additional assumption that $X$ is tight.

\begin{remark}\label{rmrk12} {\rm Suppose that $X$ is pure and tight and suppose that $\gamma: S^1\to X$ is a simplicial loop with the ratio
$|\gamma|\cdot A_X(\gamma)^{-1}$ less than the minimum of the numbers $I(X')$ where $X'\subset X$ is a proper subcomplex. Let $b: D^2\to X$ be a minimal spanning disc for $\gamma$; then $b(D^2)=X,$ i.e. $b$ is surjective. Indeed, if the image of $b$ does not contain a 2-simplex $\sigma$ then removing
it we obtain a subcomplex $X'\subset X$ with $A_{X'}(\gamma)=A_X(\gamma)$ and hence $I(X') \le I(X)\le |\gamma|\cdot A_X(\gamma)^{-1}$ contradicting the assumption on $\gamma$. }
\end{remark}


\begin{lemma}\label{lm13} If $X$ is a tight complex with $\tilde \nu(X)>1/3$ then $b_2(X)=0$.
\end{lemma}
\begin{proof} Assume that $b_2(X)\not=0$. Then there exists a minimal cycle $Z\subset X$ satisfying $\tilde \nu(Z)>1/3$. Hence, by Lemmas \ref{nonasph}, \ref{lmB1} and \ref{lmB2} we may find a 2-simplex $\sigma\subset Z\subset X$
such that $\partial \sigma$ is null-homotopic in $Z-\sigma\subset X-\sigma=X'$. Note that $X'^{(1)}=X^{(1)}$ and a simplicial curve $\gamma: S^1\to X'$ is null-homotopic in $X'$ if and only if it is null-homotopic in $X$. Besides, $A_X(\gamma) \le A_{X'}(\gamma)$ and hence
$$\frac{|\gamma|}{A_X(\gamma)}\ge \frac{|\gamma|}{A_{X'}(\gamma)},$$
which implies that $I(X) \ge I(X') >I(X)$ -- contradiction.
\end{proof}

%

\begin{lemma}\label{lneg} Given $\epsilon >0$ there exists a constant $C'_\epsilon>0$ such that for any finite pure tight connected 2-complex with $\tilde \nu(X) \ge 1/3+\epsilon$ and $L(X) \le 0$ one has $I(X) \ge C'_\epsilon$.
\end{lemma}

This Lemma is similar to Theorem \ref{uniform} but it has an additional assumption that $L(X) \le 0$. It is clear from the proof that the assumption $L(X)\le 0$ can be replaced, without altering the proof, by any assumption of the type $L(X) \le 1000$, i.e. by any specific upper bound.

\begin{proof} We show that the number of isomorphism types of complexes $X$ satisfying the conditions of the Lemma is finite; hence the statement of the Lemma follows  by setting $C'_\epsilon=\min I(X)$ and using Corollary \ref{freeproduct} which gives $I(X)>0$ (since $\pi_1(X)$ is hyperbolic) and hence $C'_\epsilon>0$.
The inequality
$$\nu(X) = \frac{1}{3} +\frac{3\chi(X)+L(X)}{3e(X)} \ge \frac{1}{3} + \epsilon$$
is equivalent to
$$e(X) \le \epsilon^{-1}\cdot (3\chi(X) +L(X)/2),$$
where $e(X)$ denotes the number of 1-simplexes in $X$.
By Lemma \ref{lm13} we have $\chi(X) =1-b_1(X) \le 1$ and using the assumption $L(X) \le 0$ we obtain
$e(X) \le \epsilon^{-1}.$
This implies the finiteness of the set of possible isomorphism types of $X$ and the result follows. 
\end{proof}

We will use a relative isoperimetric constant $I(X, X')\in \R$ for a pair consisting of a finite 2-complex
$X$ and its subcomplex $X'\subset X$; it is defined as the infimum of all ratios
$
{|\gamma|}\cdot{A_X(\gamma)}^{-1}$
where $\gamma: S^1\to X'$ runs over simplicial loops in $X'$ which are null-homotopic in $X$.
Clearly, $I(X,X')\ge I(X)$ and $I(X, X')=I(X)$ if $X'=X$.
Below is a useful strengthening of Lemma \ref{lneg}.

\begin{lemma}\label{lnegs} Given $\epsilon >0$, let $C'_\epsilon>0$ be the constant given by Lemma \ref{lneg}. Then for any finite pure tight connected 2-complex with $\tilde \nu(X) \ge 1/2+\epsilon$ and for a connected subcomplex $X'\subset X$ satisfying $L(X') \le 0$ one has $I(X,X') \ge C'_\epsilon$.
\end{lemma}

\begin{proof}
We show below that under the assumptions on $X$, $X'$ one has
\begin{eqnarray} \label{y}
I(X,X') \ge \min_Y I(Y)
\end{eqnarray}
where $Y$ runs over all subcomplexes $X'\subset Y\subset X$ satisfying $L(Y)\le 0$. Clearly, $\tilde \nu(Y)\ge 1/3+\epsilon$ for any such $Y$.
 By Lemma \ref{lm13} we have that $b_2(X)=0$ which implies that $b_2(Y)=0$. Besides, without loss of generality we may assume that
$Y$ is connected. The arguments of the proof of Lemma \ref{lneg} now apply (i.e. $Y$ may have finitely many isomorphism types, each having a hyperbolic fundamental group) and it follows that $\min_Y I(Y)\ge C'_\epsilon$ where $C'_\epsilon>0$ is a constant that only depends on $\epsilon$. Hence if (\ref{y}) holds we have $I(X,X')\ge \min_Y I(Y)\ge C'_\epsilon$ and the result follows.

Suppose that inequality (\ref{y}) is false, i.e. $I(X,X') < \min_Y I(Y)$, and consider a simplicial loop $\gamma:S^1\to X'$ satisfying $\gamma\sim 1$ in $X$ and
$|\gamma|\cdot A_X(\gamma)^{-1} <\min_Y I(Y).$ Let $\psi: D^2\to X$ be a simplicial spanning disc of minimal area.
It follows from the arguments of Ronan \cite{Ron}, that $\psi$ is non-degenerate in the following sense: for any 2-simplex $\sigma$ of $D^2$ the image $\psi(\sigma)$ is a 2-simplex
and for two distinct 2-simplexes $\sigma_1, \sigma_2$ of $D^2$ with $\psi(\sigma_1)=\psi(\sigma_2)$ the intersection $\sigma_1\cap \sigma_2$ is either $\emptyset$ or a vertex of $D^2$. In other words, we exclude {\it foldings}, i.e. situations such that $\psi(\sigma_1)=\psi(\sigma_2)$ and $\sigma_1\cap \sigma_2$ is an edge.
Consider $Z=X'\cup \psi(D^2)$. Note that $L(Z)\le 0$. Indeed, since $$L(Z)=\sum_e (2-\deg_Z(e)),$$ where $e$ runs over the edges of $Z$, we see that for $e\subset X'$,
$\deg_{X'}(e)\le \deg_Z(e)$ and for a newly created edge $e\subset \psi(D^2)$, clearly $\deg_Z(e)\ge 2$. Hence, $L(Z)\le L(X')\le 0$.
On the other hand, $A_X(\gamma)=A_Z(\gamma)$ and hence $I(Z)\le |\gamma|\cdot A_X(\gamma)^{-1}<\min_Y I(Y)$, a contradiction.
\end{proof}

The main idea of the proof of Theorem \ref{uniform} in the general case is to find a planar complex (a \lq\lq singular  surface\rq\rq)\ $\Sigma$, with one boundary component $\partial_+\Sigma$ being the initial loop and such that \lq\lq the rest of the boundary\rq\rq\, $\partial_-\Sigma$ is a \lq\lq product of negative loops\rq\rq\,  (i.e. loops satisfying Lemma \ref{lnegs}). The essential part of the proof is in estimating the area (the number of 2-simplexes) of such $\Sigma$.

\begin{proof}[Proof of Theorem \ref{uniform}]
Consider a connected tight pure 2-complex $X$ satisfying
\begin{eqnarray}\label{ass}
\tilde \nu(X) \ge \frac{1}{3} +\epsilon
\end{eqnarray}
and a simplicial prime loop $\gamma: S^1 \to X$ such that the ratio
$|\gamma|\cdot A_X(\gamma)^{-1}$ is less than the minimum of the numbers $I(X')$ for all proper subcomplexes $X'\subset X$. Consider a minimal spanning disc
$b: D^2\to X$ for $\gamma=b|_{\partial D^2}$; here $D^2$ is a triangulated disc and $b$ is a simplicial map. As we showed in Remark \ref{rmrk12}, the map $b$ is surjective.
As explained in the proof of Lemma \ref{lnegs}, due to arguments of Ronan \cite{Ron},
we may assume that $b$ has no foldings.


For any integer $i\ge 1$ we denote by $X_i\subset X$ the pure subcomplex generated by all 2-simplexes $\sigma$ of $X$ such that the preimage $b^{-1}(\sigma)\subset D^2$ contains $\ge i$ two-dimensional simplexes. One has $X=X_1\supset X_2\supset X_3\supset \dots.$ Each $X_i$ may have several connected components and we will denote by $\Lambda$ the set labelling all the connected components of the disjoint union $\sqcup_{i\ge 1} X_i$. For $\lambda\in \Lambda$ the symbol $X_\lambda$ will denote the corresponding connected component of $\sqcup_{i\ge 1} X_i$
and the symbol
$i=i(\lambda)\in \{1, 2, \dots\}$ will denote the index $i\ge 1$ such that $X_\lambda$ is a connected component of $X_i$, viewed as a subset of $\sqcup_{i\ge 1} X_i$. We endow $\Lambda$ with the following partial order: $\lambda_1\le \lambda_2$ iff $X_{\lambda_1}\supset X_{\lambda_2}$ (where $X_{\lambda_1}$ and $X_{\lambda_2}$ are viewed as subsets of $X$) and $i(\lambda_1)\le i(\lambda_2)$.

Next we define the sets
$$\Lambda^-=\{\lambda\in \Lambda; L(X_\lambda)\le 0\}$$
and
$$ \Lambda^+=\{\lambda\in \Lambda; \mbox{for any $\mu\in \Lambda$ with $\mu\le \lambda$, }\, L(X_\mu)> 0\}.$$
Finally we consider the following subcomplex  of the disk $D^2$:
\begin{eqnarray}\label{defsigma}
\Sigma' =D^2-\bigcup_{\lambda\in \Lambda^-}{\rm {Int}}(b^{-1}(X_\lambda))\end{eqnarray}
and we shall denote by $\Sigma$ the connected component of $\Sigma'$ containing the boundary circle $\partial D^2$.
%

Recall that for a 2-complex $X$ the symbol $f(X)$ denotes the number of 2-simplexes in $X$. We have
\begin{eqnarray}\label{one1}
f(D^2) = \sum_{\lambda\in \Lambda} f(X_\lambda),
\end{eqnarray}
and
\begin{eqnarray}\label{two2}
f(\Sigma) \le f(\Sigma') = \sum_{\lambda\in \Lambda^+} f(X_\lambda).
\end{eqnarray}
Formula (\ref{one1}) follows from the observation that any 2-simplex of $X=b(D^2)$ contributes to the RHS of (\ref{one}) as many units as its multiplicity (the number of its preimages under $b$). Formula (\ref{two2}) follows from (\ref{one1}) and from the fact that for a 2-simplex $\sigma$ of $\Sigma$ the image $b(\sigma)$ lies always in the complexes $X_\lambda$ with $L(X_\lambda)> 0$.

\begin{lemma} \label{lpartial} One has the following inequality
\begin{eqnarray}\label{llb}
\sum_{\lambda\in \Lambda^+}L(X_\lambda)\le |\partial D^2|.
\end{eqnarray}
\end{lemma}

See \cite{CF1}, Lemma 6.8 for the proof.

Now we continue with the proof of Theorem \ref{uniform}.
Consider a tight pure 2-complex $X$ satisfying (\ref{ass}) and a simplicial loop $\gamma: S^1\to X$ as above.
We will use the notation introduced earlier. The complex $\Sigma$ is a connected subcomplex of the disk $D^2$; it contains the boundary circle $\partial D^2$
which we will denote also by
$\partial_+\Sigma$. The closure of the complement of $\Sigma$,
$$N=\overline{D^2-\Sigma}\subset D^2$$ is a pure 2-complex. Let $N=\cup_{j\in J}N_j$ be the strongly connected components of $N$.
Each $N_j$ is PL-homeomorphic to a disc and we define
$$\partial_-\Sigma=\cup_{j\in J}\partial N_j,$$ the union of the circles $\partial N_j$ which are the boundaries of the strongly connected components of $N$.
It may happen that $\partial_+\Sigma$ and $\partial_-\Sigma$ have nonempty intersection. Also, the circles forming $\partial_-\Sigma$ may not be disjoint.

We claim that for any $j\in J$ there exists $\lambda\in \Lambda^-$ such that $b(\partial N_j)\subset X_\lambda$.
Indeed, let $\lambda_1, \dots, \lambda_r\in \Lambda^-$ be the minimal elements of $\Lambda^-$ with respect to the partial order introduced earlier. The complexes
$X_{\lambda_1}, \dots, X_{\lambda_r}$ are connected and pairwise disjoint and for any $\lambda\in \Lambda^-$ the complex $X_\lambda$ is a subcomplex
of one of the
sets $X_{\lambda_i}$, where $i=1, \dots, i$. From our definition (\ref{defsigma}) it follows that the image of the circle $b(\partial N_j)$ is contained in the union
$\cup_{i=1}^r X_{\lambda_i}$ but since $b(\partial N_j)$ is connected it must lie in one of the sets $X_{\lambda_i}$.

We may apply Lemma \ref{lnegs} to each of the circles $\partial N_j$. We obtain that each of the circles $\partial N_j$ admits a spanning discs of area
$\le K_\epsilon |\partial N_j|$, where $K_\epsilon= C'^{-1}_\epsilon$ is the inverse of the constant given by Lemma \ref{lnegs}.
Using the minimality of the disc $D^2$ we obtain that the circles $\partial N$ bound in $D^2$ several discs with
the total area
$A \le K_\epsilon\cdot |\partial_-\Sigma|.$


For $\lambda\in \Lambda^+$ one has $L(X_\lambda)\ge 1$ and $\chi(X_\lambda)\le 1$ (since $b_2(X_\lambda)=0$); in particular, $e(X_\lambda)\geq f(X_\lambda)$. Hence we have
$$4L(X_\lambda) \ge 3\chi(X_\lambda) + L(X_\lambda) \ge 3\epsilon e(X_\lambda)\ge 3\epsilon f(X_\lambda) $$
where on the second last inequality we used the inequality $\nu(X_\lambda) \ge 1/3+\epsilon$. Summing up we get
$$f(\Sigma)\le  \sum_{\lambda\in \Lambda^+} f(X_{\lambda}) \le \frac{4}{3\epsilon}\sum_{\lambda\in \Lambda^+}L(X_\lambda) \le
 \frac{4}{3\epsilon}
 |\partial D^2|.$$
The rightmost inequality  is given by Lemma \ref{lpartial}.

Next we observe, that
\begin{eqnarray}
|\partial_-\Sigma| \le 2f(\Sigma) +|\partial_+\Sigma|.
\end{eqnarray}
Therefore, we obtain
\begin{eqnarray*}
f(D^2) &\le& f(\Sigma)+A \, \le \, \frac{4}{3\epsilon} |\gamma| + K_\epsilon\cdot 2\cdot f(\Sigma) + K_\epsilon |\gamma| \\
&\le &
\left(\frac{4}{3\epsilon}(1+2K_\epsilon) +K_\epsilon\right)\cdot |\gamma|,
\end{eqnarray*}
implying
\begin{eqnarray}
I(X) \ge \frac{3\epsilon}{4+8K_\epsilon+3\epsilon K_\epsilon}.
\end{eqnarray}
This completes the proof of Theorem \ref{uniform}.
\end{proof}


\bibliographystyle{amsalpha}

\begin{thebibliography}{99}

%
\bibitem{Adams} J. F. \ Adams,  \textit{A new proof of a theorem of W. H. Cockcroft}, J. London Math. Soc.  30  (1955), 482–488. 

\bibitem{ALLM} L. \ Aronshtam, N. \ Linial, T. \ {\L}uczak, R. \ Meshulam, \textit{Vanishing of the top homology of a random complex}, Discrete \& Computational Geometry
49(2013), pp 317--334.

\bibitem{ALS} S. Antoniuk, T. {\L}uczak, J.  \'{S}wi\c{a}tkowski, \textit{Random triangular groups at density 1/3}, preprint arXiv:1308.5867v2.


\bibitem{BHK} E.\ Babson, C.\ Hoffman, M.\ Kahle, 
{\it The fundamental group of random $2$-complexes}, J. Amer. Math. Soc. 24 (2011), 1-28. 
See also the latest archive version arXiv:0711.2704 revised on 20.09.2012. 

\bibitem {Babson} E. Babson, Fundamental groups of random clique complexes, arXiv:1207.5028v2

\bibitem{BB} M. Bestvina, N. Brady, \textit{Morse theory and finiteness properties of groups,}
Invent. Math.  129  (1997),  no. 3, 445–-470. 

\bibitem{B} W.\ A.\ Bogley, \textit{J.H.C. Whitehead's asphericity question}, in: "Two-dimensional Homotopy and Combinatorial Group Theory", eds. 
C. Hog-Angeloni, A. Sieradski and W. Metzler, LMS Lecture Notes 197, Cambridge Univ Press (1993), 309-334. 


\bibitem{Brown} W. Brown, \textit{Enumeration of triangulations of the disk}, Proc. London Math. Soc. (3)14 (1964), 746-768. 


\bibitem{CC} W.H. Cockcroft, \textit{On two-dimensional aspherical complexes}, Proc. London Math. Soc. {\bf {4}}(1954), 375-384. 

\bibitem{CCFK} D. Cohen, A.E. Costa, M. Farber, T. Kappeler, \textit{Topology of random 2-complexes,} 
Journal of Discrete and Computational Geometry, {\bf 47}(2012), 117-149.

\bibitem{CF} A.E. Costa, M. Farber, \textit{The asphericity of random 2-dimensional complexes}, to appear in \textit{Random Structures and Algorithms,}
 preprint 
arXiv:1211.3653v1

\bibitem{CF1} A.E. Costa, M. Farber, \textit{Geometry and topology of random 2-complexes},  arXiv:1307.3614. 

\bibitem{ER} P.\ Erd\H{o}s, A.\ R\'enyi, {\it On the evolution of 
random graphs}, Publ.\ Math.\ Inst.\ Hungar.\ Acad.\ Sci.\ {\bf 5}
(1960), 17--61.

\bibitem{EG} S. Eilenberg and T. Ganea, \textit{On the Lusternik-Schnirelmann category of abstract groups. }
Ann. of Math. (2)  65  (1957), 517–-518. 

\bibitem{Gromov83} M. Gromov, \textit{Filling Riemannian manifolds}, J. Differential Geometry, 18(1983), 1--147. 

\bibitem{Gromov87} M. Gromov, \textit{Hyperbolic groups}, in Essays in group theory, ed.
S. M. Gersten, Springer (1987), 75-265.

\bibitem{HR} N. Hartdfield and G. Ringel, \textit{Clean triangulations}, Combinatorica 11(2)(1991), pp. 145 - 155. 

\bibitem{Hatcher} A. Hatcher, \textit{Alegbraic Topology}, Cambridge University Press, Cambridge 2002.


\bibitem{JLR} S. Janson, T. {\L}uczak, A. Ruci\'nski, \textit{Random graphs}, Wiley-Intersci. Ser. Discrete Math. Optim., Wiley-Interscience, New York, 2000.

\bibitem {Kahle1} M. Kahle, Topology of random clique complexes, Discrete Math. 309 (2009), no. 6,
1658 -- 1671. MR MR2510573
\bibitem {Kahle3} M. Kahle, Sharp vanishing thresholds for cohomology of random flag complexes, 
Ann. of Math. (2)  179  (2014),  no. 3, 1085–1107.

\bibitem{KM} M. Kahle and E. Meckes, Limit theorems for Betti numbers of random simplicial
complexes, Homology Homotopy Appl.  15  (2013),  no. 1, 343 -- 374.

\bibitem {KP} M. Kahle and B. Pittel, Inside the critical window for cohomology of random k-complexes, 
To appear in  Random Structures Algorithms, arXiv:1301.1324.

\bibitem{Ksurvey} M. Kahle, \textit{Topology of random simplicial complexes: a survey},  To appear in AMS Contemporary Volumes in Mathematics. Nov 2014. arXiv:1301.7165.









\bibitem{Ron} M.A. Ronan, \textit{On the second homotopy group of certain simplicial complexes and some combinatorial applications}, Quart. J. Math. {\bf {32}}(1981), 225 - 233. 


\bibitem{R} S. Rosenbrock, \textit{The Whitehead Conjecture - an overview}, Siberian Electronoc Mathematical Reports, {\bf 4}(2007), 440-449. 

\bibitem{Sw} R. G. Swan, \lq\lq Groups of cohomological dimension one\rq\rq. Journal of Algebra  {\bf 12} (1969), 585–610







\end{thebibliography}

\vskip 2cm
A. Costa and M. Farber: 

School of Mathematical Sciences, Queen Mary University of London, London E1 4NS
\vskip 0.7cm

D. Horak: Theoretical Systems Biology, Impreial College, London, SW7 2AZ

\end{document}